\documentclass[a4paper,11pt,reqno]{amsart}
\usepackage[letterpaper, margin=1 in]{geometry}
\usepackage{caption} 
\usepackage{amsthm}
\usepackage{bm}

\newtheorem{theorem}{Theorem}[section]

\newtheorem{lemma}[theorem]{Lemma}

\newtheorem{corollary}[theorem]{Corollary}

\newtheorem{proposition}[theorem]{Proposition}

\theoremstyle{definition}
\newtheorem{definition}[theorem]{Definition}
\newtheorem{example}[theorem]{Example}
\newtheorem{remark}[theorem]{Remark}

\usepackage{adjustbox}
\usepackage{float}
\usepackage{amssymb,amsfonts,amsthm,amsmath,amsopn,amstext,amscd,latexsym,color}
\usepackage[all]{xy}
\usepackage{mathtools}

\usepackage{enumerate}

\xyoption{web}

\usepackage[utf8]{inputenc}

\newcommand{\bbA}{{\mathbb A}}

\newcommand{\bbG}{{\mathbb G}}

\newcommand{\bbN}{{\mathbb N}}

\newcommand{\bbQ}{{\mathbb Q}}

\newcommand{\bbZ}{{\mathbb Z}}
\usepackage{diagbox}
\usepackage{longtable}
\DeclareMathOperator{\Pic}{Pic}

\DeclareMathOperator{\lcm}{lcm}

\usepackage{hyperref}

\usepackage{tikz-cd}

\usepackage{mathtools}

\usepackage{longtable}
\usepackage{verbatim}

\usepackage[OT2,T1]{fontenc}
\DeclareSymbolFont{cyrletters}{OT2}{wncyr}{m}{n}
\DeclareMathSymbol{\Sha}{\mathalpha}{cyrletters}{"58}

\usepackage{graphics}

\newcommand{\Q}{\mathbb{Q}}
\newcommand{\Z}{\mathbb{Z}}

\usepackage{enumitem}

\numberwithin{equation}{section}
\newcommand{\code}[1]{\texttt{#1}}

\DeclareMathOperator{\Gal}{Gal}
\DeclareMathOperator{\Br}{Br}
\DeclareMathOperator{\im}{Im}
\DeclareMathOperator{\Ker}{Ker}
\DeclareMathOperator{\Coker}{Coker}

\DeclareMathOperator{\Hom}{Hom}

\DeclareMathOperator{\Res}{Res}

\DeclareMathOperator{\Cor}{Cor}

\author{\sc André Macedo and Rachel Newton}

\address{André Macedo\\
University of Reading\\
Department of Mathematics and Statistics\\ 
Pepper Lane\\
Whiteknights\\
Reading RG6 6AX\\
UK}
   \email{c.a.v.macedo@pgr.reading.ac.uk}
\urladdr{https://sites.google.com/view/andre-macedo}

\address{Rachel Newton\\
University of Reading\\
Department of Mathematics and Statistics\\ 
Pepper Lane\\
Whiteknights\\
Reading RG6 6AX\\
UK}
   \email{r.d.newton@reading.ac.uk}
\urladdr{https://racheldominica.wordpress.com/}

\title[Explicit methods for the HNP and applications to $A_n$ and $S_n$ extensions]{Explicit methods for the Hasse norm principle and applications to $A_n$ and $S_n$ extensions}

\begin{document}

\maketitle

\begin{abstract}
Let $K/k$ be an extension of number fields.
We describe theoretical results and computational methods for calculating the
obstruction to the Hasse norm principle for $K/k$ and the defect of weak approximation for the norm one torus $R^1_{K/k}\bbG_m$.
We apply our techniques to give explicit and computable formulae for the obstruction to the Hasse norm principle and the defect of weak approximation when the normal closure of $K/k$ has symmetric or alternating Galois group.
\end{abstract}

\tableofcontents
\section{Introduction}\label{sec:intro}

In this paper we study a local-global principle known as the \emph{Hasse norm principle} (HNP). Let $K/k$ be an extension of number fields with associated id\`{e}le groups $\bbA^*_K$ and $\bbA^*_k$. The norm map $N_{K/k} : K^* \to k^*$ extends to an id\`{e}lic norm map $N_{K/k} : \bbA^*_K \to \bbA^*_k$. The HNP is said to hold for $K/k$ if the so-called \emph{knot group} \[\mathfrak{K}(K/k)=(k^*\cap N_{K/k}\bbA_K^*)/N_{K/k}K^*\]
is trivial, i.e.~if being a norm everywhere locally is equivalent to being a global norm from $K/k$. For example, if $N/k$ is the normal closure of $K/k$, then the HNP holds for $K/k$ in the following cases:

\begin{enumerate}[label=(\Roman{*})]
\item \label{HNT}(The Hasse norm theorem:) $N=K$ and $\Gal(K/k)$ is cyclic  \cite{Hasse};
\item \label{prime}$[K:k]$ is prime \cite{Bartelsprime};
\item \label{dihedral} $[K:k]=n$ and $\Gal(N/k)\cong D_n$ is dihedral of order $2n$ \cite{Bartels};
\item \label{Sn} $[K:k]=n$ and $\Gal(N/k)\cong S_n$ \cite{VoKu} (see also \cite{Vosk2});
\item \label{An} $[K:k]=n\geq 5$ and $\Gal(N/k)\cong A_n$ \cite{macedo}.
\end{enumerate}
Biquadratic extensions provide the simplest setting in which the HNP can fail. For example, $3$ is everywhere locally a norm from $\mathbb{Q}(\sqrt{-3},\sqrt{13})/\mathbb{Q}$, but not a global norm \cite{Hasse}.

The HNP also has a geometric interpretation: the knot group $\mathfrak{K}(K/k)$ is identified with the Tate--Shafarevich group $\Sha(T)$ of the norm one torus $T=R^1_{K/k}\mathbb{G}_m$ defined by the following exact sequence of algebraic groups over $k$:
\[1\to R^1_{K/k}\mathbb{G}_m\to R_{K/k}\mathbb{G}_{m}\to \mathbb{G}_{m,k}\to 1\]
where $R_{K/k}\mathbb{G}_{m}$ denotes the Weil restriction of $\mathbb{G}_m$ from $K$ to $k$. The HNP holds for $K/k$ if and only if the Hasse principle holds for all principal homogeneous spaces for $R^1_{K/k}\mathbb{G}_m$. 

\emph{Weak approximation} is said to hold for a torus $T$ over $k$ if its $k$-points are dense in the product of its points over all completions of $k$; in other words if $A(T) = 0$, where
$A(T) =\prod_v T(k_v)/\overline{T(k)}$
and $\overline{T(k)}$ denotes the closure of $T(k)$ in $\prod_v T(k_v)$ with respect to the
product topology.
The following exact sequence, due to Voskresenski\u{\i} in \cite{Vosk}, ties together weak approximation for a torus $T$ and the Hasse principle for principal homogeneous spaces for $T$: 
\begin{equation}\label{eq:Vosk}
    0  \to A(T)  \to \operatorname{H}^1(k,\Pic \overline{X})^{\sim} \to \Sha (T) \to 0.
\end{equation}
Here, $X$ denotes a smooth compactification of $T$ and $\operatorname{H}^1(k,\Pic \overline{X})^{\sim}=\Hom(\operatorname{H}^1(k,\Pic \overline{X}),\bbQ/\bbZ)$. Note that the Hochschild--Serre spectral sequence gives an isomorphism $\Br X/\Br_0 X\cong \operatorname{H}^1(k,\operatorname{Pic}\overline{X})$, where $\Br_0 X=\im(\Br k\to \Br X)$. 
While results of Colliot-Th\'{e}lène and Sansuc (see e.g.~Theorem~\ref{thmcs}) enable computation of the invariant $\operatorname{H}^1(k,\Pic \overline{X})$, and a result of Tate (see Theorem~\ref{Tate}) does the same for the Tate--Shafarevich group, actually computing these groups in practice can be challenging. In this paper we give new methods for computing these invariants in the case of norm one tori associated to extensions of number fields.

\paragraph{\bf Set-up.} Except where stated otherwise, our assumptions throughout the rest of the paper will be as follows. Let $T=R^{1}_{K/k} \mathbb{G}_m$ and let $X$ denote a smooth compactification of $T$. Let $L/k$ be a Galois extension containing $K/k$ and set $G=\Gal(L/k)$ and $H=\Gal(L/K)$.

In addition to general techniques from the arithmetic of algebraic tori, our work makes use of a quotient of the knot group called the `\emph{first obstruction to the HNP for $K/k$ corresponding to the tower $L/K/k$}' defined by Drakokhrust and Platonov in \cite{DP} as
$$\mathfrak{F}(L/K/k)=(k^*\cap N_{K/k}\bbA_K^*)/(k^*\cap N_{L/k}\bbA_L^*)N_{K/k}K^*,$$

\noindent i.e.~as the cokernel of the natural map $\mathfrak{K}(L/k) \to\mathfrak{K}(K/k)$. As shown in \cite{DP}, the first obstruction to the HNP in a tower of number fields admits a purely group-theoretic description in terms of the relevant local and global Galois groups, see Theorem~\ref{thm1DP}.

Let $X_0$ be a smooth compactification of the torus $R^1_{L/k} \bbG_m$. The map $N_{L/K}:R^1_{L/k} \bbG_m\to R^1_{K/k} \bbG_m$ induces a canonical map $f_{L/K}:\operatorname{H}^1(k, \Pic\overline{X_0})^{\sim} \to \operatorname{H}^1(k, \Pic\overline{X})^{\sim}$, see \cite[\S1.2.2.]{BK00}. In order to study the birational invariant $\operatorname{H}^1(k, \Pic\overline{X})$, we introduce an object called the `\emph{unramified cover of the first obstruction to the HNP for $K/k$ corresponding to the tower $L/K/k$}' defined as
$$\mathfrak{F}_{nr}(L/K/k)=\Coker(f_{L/K}).$$

\noindent In similar fashion to the first obstruction to the HNP, its unramified cover $\mathfrak{F}_{nr}(L/K/k)$ also admits an explicit group-theoretic description:

\begin{theorem}\label{thm:unramified_description_intro}
There is a canonical isomorphism 
\[\mathfrak{F}_{nr}(L/K/k) = (H\cap [G,G])/\Phi^G(H),\] where $\Phi^G(H)$ denotes the focal subgroup of $H$ in $G$, see Definition~\ref{def:focal}.
\end{theorem}

As a corollary, one can compute the $p$-primary parts of the knot group, the invariant $\operatorname{H}^1(k, \Pic\overline{X})$, and the defect of weak approximation for all but finitely many primes $p$. In what follows, let $\mathfrak{F}(G,H)=(H\cap [G,G])/\Phi^G(H)$ and write $M_{(p)}$ for the $p$-primary part of an abelian group $M$.

\begin{corollary}\label{an_sn_p_part_intro}
If $p$ is a prime such that $\operatorname{H}^3(G,\Z)_{(p)} = 0$, then

\begin{enumerate}[label=(\roman{*})]
    \item $\mathfrak{K}(K/k)_{(p)} = \mathfrak{F}(L/K/k)_{(p)}$;
    \item $\operatorname{H}^1(k,\Pic \overline{X})^{\sim}_{(p)} = \mathfrak{F}(G,H)_{(p)}$; 
    \item $A(T)_{(p)}=\ker\left(\mathfrak{F}(G,H)_{(p)}\to \mathfrak{F}(L/K/k)_{(p)}\right)$, where the map $\mathfrak{F}(G,H)\to \mathfrak{F}(L/K/k)$ is a natural surjection, see Section~\ref{sec:1obs}.
\end{enumerate}  
\end{corollary}

We now restrict our focus to extensions with normal closure having Galois group isomorphic to $A_n$ or $S_n$. Our first main theorem enables a purely computational analysis of the HNP and weak approximation for extensions in this family.

\begin{theorem}\label{main0}

Suppose that $G$ is isomorphic to $A_n$ or $S_n$ for some $n \geq 4$ and $G\not\cong A_6, A_7$. 
Then
\[\mathfrak{K}(K/k)=\begin{cases}
\mathfrak{F}(L/K/k)\textrm{, if $|H|$ is even;}\\
\mathfrak{F}(L/K/k)\times \mathfrak{K}(L/k)  \textrm{, if $|H|$ is odd,}
\end{cases}\]
and 
\[\operatorname{H}^1(k,\Pic \overline{X})^\sim=\begin{cases}
\mathfrak{F}_{nr}(L/K/k), \textrm{if $|H|$ is even;}\\
  \mathfrak{F}_{nr}(L/K/k)\times\bbZ/2\textrm{, if $|H|$ is odd.}
\end{cases}\]
\end{theorem}

Theorem~\ref{Tate}, due to Tate, shows that the knot group of the Galois extension $L/k$ is dual to $\Ker(\operatorname{H}^3(G,\bbZ)\to \prod_{v}\operatorname{H}^3(D_v,\bbZ))$, where $D_v$ denotes the decomposition group at a place $v$ of $k$. Note that this kernel only depends on the decomposition groups at the ramified places, since if $v$ is unramified then $D_v$ is cyclic and hence $\operatorname{H}^3(D_v,\bbZ)=0$. In the setting of Theorem~\ref{main0} we obtain an algorithm that takes as inputs $G$, $H$ and the decomposition groups at the ramified places of $L/k$ and gives as its outputs the knot group $\mathfrak{K}(K/k)$, the invariant $\operatorname{H}^1(k,\Pic \overline{X})$, and the defect of weak approximation $A(T)$ for $T=R^1_{K/k}\bbG_m$.

Using Theorem~\ref{main0} we also characterize the possible isomorphism classes of the group $\operatorname{H}^1(k,\Pic \overline{X})$:

\begin{theorem}
\label{thm:an_sn_options}
\begin{enumerate}[label=(\roman{*})]
\item\label{options_1} For $G\cong S_n$ the invariant $\operatorname{H}^1(k,\Pic \overline{X})$ is an elementary abelian $2$-group. Moreover, every possibility for $\operatorname{H}^1(k,\Pic \overline{X})$ is realised: given an elementary abelian $2$-group $A$, there exists $n\in\bbN$ and an extension of number fields $K/k$ whose normal closure has Galois group $S_n$ such that $\operatorname{H}^1(k,\Pic \overline{X})\cong A$, where $X$ is a smooth compactification of $R^1_{K/k} \bbG_m$.
\item For $G\cong A_n$ the invariant $\operatorname{H}^1(k,\Pic \overline{X})$ is either isomorphic to $C_3$, $C_6$ or an elementary abelian $2$-group. Again, every possibility for $\operatorname{H}^1(k,\Pic \overline{X})$ is realised.
\end{enumerate}
\end{theorem}

\begin{remark}
The statement of Theorem~\ref{thm:an_sn_options} also holds if one replaces $\operatorname{H}^1(k,\Pic \overline{X})$ by $\mathfrak{K}(K/k)$ or $A(T)$, see Proposition~\ref{prop:an_sn_options_knotgp}.
\end{remark}

Theorems~\ref{main0} and \ref{thm:an_sn_options} can be combined to obtain more precise information, as demonstrated in Corollary~\ref{cor:thorough} and Example~\ref{eg} below.

\begin{corollary}\label{cor:thorough}
Retain the assumptions of Theorem~\ref{main0} and, for $p$ prime, let $H_p$ denote a Sylow $p$-subgroup of $H$. Then $\operatorname{H}^1(k,\Pic \overline{X})_{(p)}=0$ for all primes $p>3$, $\operatorname{H}^1(k,\Pic \overline{X})_{(3)}=0$ if $G\cong S_n$,
\[\operatorname{H}^1(k,\Pic \overline{X})^\sim_{(2)}=\begin{cases}\mathfrak{F}(G,H)[2]\cong\mathfrak{F}(G,H_2) & \textrm{if $|H|$ is even;}\\
\Z/2 & \textrm{if $|H|$ is odd,}
\end{cases}
\]
and if $G\cong A_n$ then
\[\operatorname{H}^1(k,\Pic \overline{X})^\sim_{(3)}=\mathfrak{F}(G,H)[3]\cong\mathfrak{F}(G,H_3).\]
In particular, if $3\nmid |H|$ then $\operatorname{H}^1(k,\Pic \overline{X})$ is $2$-torsion.
\end{corollary}

\begin{example}\label{eg}
Suppose that $G \cong S_n$ and $|H|$ is odd. Then $\operatorname{H}^1(k,\Pic \overline{X}) = \Z/2$ and ${\mathfrak{K}(K/k) = \mathfrak{K}(L/k)}$. The same conclusion holds for $G\cong A_n$ under the stronger assumption that $|H|$ is coprime to $6$.
\end{example}

As a further application of Theorem~\ref{main0}, one can obtain conditions on the decomposition groups determining whether the HNP and weak approximation hold in $A_n$ and $S_n$ extensions. In Propositions~\ref{main} and \ref{weakap}, we exhibit such a characterization for $n=4$ or $5$, when these local conditions are particularly simple.

\begin{proposition}\label{main}

Suppose that $G$ is isomorphic to $A_4, A_5, S_4$ or $S_5$. Then $\mathfrak{K}(K/k)\hookrightarrow C_2$ and
\begin{enumerate}[label=(\roman{*})]
\item\label{main odd} if $|H|$ is odd, then $\mathfrak{K}(K/k)=1\iff \exists\ v$ such that 
$V_4\hookrightarrow D_v$;
\item\label{main ii} if $\exists$ $C\leq  H$ generated by a double transposition with $[H:C]$ odd, then $\mathfrak{K}(K/k)=1\iff \exists\ v$ such that $D_v$ contains a copy of $V_4$ generated by two double transpositions;
\item \label{main HNP} in all other cases, $\mathfrak{K}(K/k)=1$.
\end{enumerate}
\end{proposition}

\begin{proposition}
\label{weakap}
Retain the assumptions of Proposition~\ref{main}.
Then \[\operatorname{H}^1(k,\Pic \overline{X})=\begin{cases}
\bbZ/2 & \textrm{in cases (i) and (ii) of Proposition~\ref{main}};\\
0 & \textrm{otherwise}.
\end{cases}\]
Therefore, in cases (i) and (ii) of Proposition~\ref{main}, weak approximation holds for $R^{1}_{K/k} \mathbb{G}_m$ if and only if the HNP fails for $K/k$. In all other cases, weak approximation holds for $R^{1}_{K/k} \mathbb{G}_m$.
\end{proposition}

For the sake of completeness, we also provide criteria for the validity of the HNP when $G \cong A_6$ or $A_7$ (the two groups not addressed by Theorem \ref{main0}), see Propositions~\ref{thm:A6A7} and \ref{thm:waconditionsA6}. The proof uses the first obstruction to the HNP, along with various tricks involving moving between subextensions as detailed in Section~\ref{sec:upanddown}.

Our motivation for providing explicit local conditions for the failure of the HNP is to enable a statistical analysis of the HNP and weak approximation for norm one tori in families of extensions of number fields. 
Such an analysis was carried out for extensions of a number field $k$ with fixed abelian Galois group by the second author together with Frei and Loughran in \cite{FLN} (ordering by discriminant) and \cite{FLN2} (ordering by conductor). One consequence of their results is that the HNP fails for $0\%$ of biquadratic extensions of $k$.
In the case $k=\mathbb{Q}$, this was refined to an asymptotic formula for the number of biquadratics failing the HNP (ordered by discriminant) by Rome in \cite{Rome}. 

Having dealt with the $V_4$ case and noting that the HNP holds for all $C_4$, $D_4$ and $S_4$ quartics (see \ref{HNT}, \ref{dihedral} and \ref{Sn}),
if one wants to fully understand the frequency of failure of the HNP for quartics with fixed Galois group, 
there is one remaining family to tackle: namely $A_4$ quartics. Counting $A_4$ quartics may be beyond current capabilities but the following corollary of Proposition~\ref{main} gives hope that one may be able to exploit results about biquadratic extensions to bound the number of $A_4$ quartics for which the HNP fails.

\begin{corollary}\label{cor:A4}
Let $K/k$ be a quartic extension of number fields with normal closure $L/k$ such that $G=\Gal(L/k)$ is isomorphic to $A_4$. 
Let $F$ be the fixed field of the copy of $V_4$ in $G$. Then 
\[\mathfrak{K}(K/k)\cong\mathfrak{K}(L/k)\cong \mathfrak{K}(L/F).\]
In particular, the HNP holds for $K/k$ if and only if it holds for the biquadratic extension $L/F$. Likewise, weak approximation holds for $K/k$ if and only if it holds for $L/F$.

\end{corollary}

The first statistical study of the HNP in a family of extensions with fixed non-abelian Galois group is carried out by the first author in \cite{MacedoD4}, where he shows that the HNP fails for $0\%$ of $D_4$ octics ordered by an Artin conductor. The present paper provides the algebraic input required to study the statistics of the HNP and weak approximation in several more families of non-abelian, and even non-Galois, number fields -- such as $S_4$ octics, for example.
This future work will capitalize on recent advances in counting within families of number fields, see e.g.~\cite{ASVW17, Bhargava,Bhargava_quintics,BST13,BV15,EPW17,FW18,HSV18,PTBW17,Woo17},
and contribute to the ongoing rapid progress in the area of rational points and failures of local-global principles in families of varieties. See \cite{Browning} for a survey of recent developments in this area. 
  
Although counting degree $n > 4$ extensions of number fields with bounded discriminant may be out of reach at present, there are very precise conjectures for the number of such extensions. Namely, the weak Malle conjecture on the distribution of number fields (see \cite{Malle1}) predicts that the number $N(k,G,X)$ of degree $n$ extensions $K$ of a number field $k$ with Galois group $G$ and $|N_{k/\Q}(\operatorname{Disc}_{K/k})|\leq X$ satisfies 
\begin{equation}\label{malle_estimates}
    X^{\frac{1}{\alpha(G)}} \ll N(k,G,X) \ll X^{\frac{1}{\alpha(G)}+\epsilon},
\end{equation}

\noindent where $\alpha(G)=\min\limits_{g \in G \setminus \{1\}}  \{\operatorname{ind}(g) \}$ and $\operatorname{ind}(g)$ equals $n$ minus the number of orbits of $g$ on $\{1,\dots,n\}$. Using a computational method developed by Hoshi and Yamasaki to determine ${\operatorname{H}^1(k, \Pic \overline{X})}$ (see Section \ref{sec:comp_met}), we obtain the following consequence of this conjecture:

\begin{theorem}\label{thm:malle_thm}
Fix a number field $k$ and an integer $n \leq 15$ with $n \neq 8,12$. Suppose that Conjecture~\eqref{malle_estimates} holds for every transitive subgroup $G \leq S_n$. Then

\begin{enumerate}[label=(\roman{*})]
    \item the HNP holds for 100\% of degree $n$ extensions over $k$, when ordered by discriminant;
    \item weak approximation holds for 100\% of norm one tori of degree $n$ extensions over $k$, when ordered by discriminant of the associated extension.
\end{enumerate}  
\end{theorem}

In fact, the assertions of Theorem~\ref{thm:malle_thm} remain true if one only assumes Conjecture~\eqref{malle_estimates} for a few transitive subgroups of $S_n$, see Remark \ref{malle_proof_rmk}\ref{malle_proof_rmk_item3}. An analysis of the invariant ${\operatorname{H}^1(k, \Pic \overline{X})}$ for extensions $K/k$ of degree $n \leq 15$ has also recently been carried out independently by Hoshi, Kanai and Yamasaki in \cite{HoshiYamasaki1} and \cite{HoshiYamasaki2}. In these works, the computation of $\operatorname{H}^1(k, \Pic \overline{X})$ for such extensions (which in the present paper happened behind the scenes of the proof of Theorem~\ref{thm:malle_thm}) is made explicit and, additionally, necessary and sufficient conditions for the vanishing of $\mathfrak{K}(K/k)$ are given.

\label{page_gh_extension}

In order to obtain asymptotic formulae for the number of extensions satisfying certain conditions, it is often necessary to first show the existence of at least one such extension, see \cite[Theorem~1.7]{FLN}, for example. 
Our next result addresses this issue. 
Let $G$ be a finite group and $H$ a subgroup of $G$. We define a $(G,H)$-extension of a number field $k$ to be an extension $K/k$ for which there exists a Galois extension $L/k$ containing $K/k$ such that $\Gal(L/k) \cong G$ and $\Gal(L/K) \cong H$. We write $F_{G/H}$ for a flasque module in a flasque resolution of the Chevalley module $J_{G/H}$, see Section~\ref{sec:arithmetic_tori}.

\begin{theorem}\label{examples_general_intro}
Let $G$ be a finite group and $H$ a subgroup of $G$. Then
\begin{enumerate}[label=(\roman{*})]
    \item there exist a number field $k$ and a $(G,H)$-extension of $k$ satisfying the HNP and, furthermore, if $\operatorname{H}^1(G,F_{G/H})\neq 0$ then weak approximation fails for the norm one torus associated to this extension;
    \item\label{existwa} there exist a number field $k$ and a $(G,H)$-extension of $k$ whose norm one torus satisfies weak approximation and, furthermore, if $\operatorname{H}^1(G,F_{G/H})\neq 0$ then this extension fails the HNP.
\end{enumerate}
\end{theorem}

 The condition $\operatorname{H}^1(G,F_{G/H})\neq 0$ in Theorem~\ref{examples_general_intro} is necessary because for a $(G,H)$-extension $K/k$ with $X$ a smooth compactification of $R^1_{K/k}\bbG_m$, one has $\operatorname{H}^1(k,\operatorname{Pic}\overline{X})= \operatorname{H}^1(G,F_{G/H})$. This is due to Colliot-Thélène and Sansuc (see Theorem~\ref{thmcs}).
 
 It is interesting to compare Theorem~\ref{examples_general_intro} with \cite[Theorem 1.3]{FLN}, where the authors prove existence of Galois extensions failing the HNP with prescribed solvable Galois group $G$ and base field $k$. Here we avoid the restriction on $G$ but lose control of the base field which, in both cases of Theorem \ref{examples_general_intro}, may be of quite large degree over $\Q$. In Section~\ref{sec:eg}, we give explicit examples of extensions of number fields illustrating all cases of Proposition~\ref{main}. The examples of field extensions for which the HNP holds all have base field $\bbQ$, and the examples for which the HNP fails have base fields that are at most quadratic extensions of $\bbQ$.

\subsection{Structure of the paper}
Section~\ref{sec:arithmetic_tori} contains some relevant background material concerning the arithmetic of algebraic tori.
In Section~\ref{sec:upanddown} we gather results that allow one to transfer information regarding the HNP from a field extension to its subextensions and vice versa. We also give analogues of these results for weak approximation on the associated norm one tori. In Section~\ref{sec:1obs} we prove Theorem~\ref{thm:unramified_description_intro} and Corollary~\ref{an_sn_p_part_intro}. In Section~\ref{sec:gen_rep} we introduce generalized representation groups and outline work of Drakokhrust which uses these groups to describe the invariant $\operatorname{H}^1(k,\operatorname{Pic}\overline{X})$ occurring in Voskresenski\u{\i}'s exact sequence~\eqref{eq:Vosk}.
In Section~\ref{sect_applications} we apply our results to extensions whose normal closure has Galois group $A_n$ or $S_n$, proving Theorems~\ref{main0} and \ref{thm:an_sn_options}, Corollary~\ref{cor:thorough}, Propositions~\ref{main} and \ref{weakap}, Corollary~\ref{cor:A4} and Theorem~\ref{thm:malle_thm}. 
In Section~\ref{sec:eg} we prove Theorem~\ref{examples_general_intro} and give examples of successes and failures of the HNP in all cases covered by Proposition~\ref{main}.

\subsection{Notation}
Given a number field $k$ and a Galois extension $L/k$, we use the following notation:

\begin{longtable}{p{1cm} p{13cm}}
$\overline{k}$ & an algebraic closure of $k$\\
$\mathbb{A}_{k}^*$ & the id\`{e}le group of $k$\\
$\mathcal{O}_k$ &the ring of integers of $k$\\
$\Omega_k$ & the set of all places of $k$\\
$L_v$ & the completion of $L$ at some choice of place above $v \in \Omega_k$\\
$D_v$ & the Galois group of $L_v/k_v$
\end{longtable}

Given a field $K$, a variety $X$ over $K$ and an algebraic $K$-torus $T$, we use the following notation:

\begin{longtable}{p{1.5cm} p{13cm}}
$\mathbb{G}_{m,K}$ & the multiplicative group $\operatorname{Spec}(K[t,t^{-1}])$ of $K$ (when $K$ is clear from the context we omit it from the subscript)\\
$X_L$ & the base change $X \times_{K} L$ of $X$ to a field extension $L/K$ \\
$\overline{X}$ & the base change of $X$ to an algebraic closure of $K$ \\
$\operatorname{Pic}X$ & the Picard group of $X$\\
$\widehat{T}$ & the character group $\operatorname{Hom}(\overline{T},\mathbb{G}_{m,\overline{K}})$ of $T$\\
$R_{K/k} T$ & the Weil restriction of $T$ to a subfield $k$ of $K$\\
$R^{1}_{K/k} \mathbb{G}_m$ & the kernel of the norm map $ N_{K/k} : R_{K/k}\mathbb{G}_{m} \to \mathbb{G}_{m,k}$
\end{longtable}

For an algebraic torus $T$ defined over a number field $k$, we denote its Tate--Shafarevich group and defect of weak approximation by
\begin{equation*} \label{def:Sha}
	\Sha(T) := \Ker \left( \mathrm{H}^1(k,T)  \to \prod_{v \in \Omega_k} \mathrm{H}^1(k_v,T) \right) \text{ and } A(T) := \left(\prod\limits_{v\in \Omega_k}T(k_v)\right)/\overline{T(k)},
\end{equation*}
 respectively.

 Given a finite group $G$, a $G$-module $A$, an integer $q$ and a prime number $p$, we use the following notation:

\begin{longtable}{p{1.5cm} p{13cm}}
$|G|$ & the order of $G$\\
$\exp(G)$ & the exponent of $G$\\
$[G,G]$ & the derived subgroup of $G$\\
$G^{\sim}$ & the $\Q / \Z$-dual $\Hom(G,\Q/\Z)$ of $G$\\

$G_p$ & a Sylow $p$-subgroup of $G$\\
$\hat{\operatorname{H}}^q(G,A)$ & the Tate cohomology group\\
$\Sha_{\omega}^q(G,A)$ & the kernel of the restriction map $\hat{\operatorname{H}}^q(G,A) \xrightarrow[]{\operatorname{Res}} \prod_{g \in G} \hat{\operatorname{H}}^q(\langle g \rangle,A)$.\\
\end{longtable}

For $x,y \in G$ we adopt the convention $[x,y]=x^{-1}y^{-1}xy$ and $x^y=y^{-1}xy$. If $G$ is abelian and $d \in \Z_{> 0}$, we use the following notation:

\begin{longtable}{p{1.5cm} p{13cm}}
$G[d]$ & the $d$-torsion of $G$\\
$G_{(d)}$ & the $d$-primary part of $G$.\\

\end{longtable}

We often use `$=$' to indicate a canonical isomorphism between two objects. 

\subsection{Acknowledgements}
We are grateful to Manjul Bhargava for conversations that motivated our work on this topic, to Jean-Louis Colliot-Th\'{e}l\`{e}ne for useful discussions which led to a cleaner proof of Corollary~\ref{cor:mid_gp_general_intro} and to the anonymous referees for valuable suggestions and for pointing out the geometric interpretation of the first obstruction to the HNP (Theorem~\ref{thm:unramified_description_intro}) which improved several results and proofs in the paper. We thank Levent Alpoge, Henri Cohen, Valentina Grazian, Samir Siksek, Anitha Thillaisundaram and Rishi Vyas for helpful conversations. 
André Macedo is supported by the Portuguese Foundation of Science and Technology (FCT) via the doctoral scholarship SFRH/BD/117955/2016. Rachel Newton is supported by EPSRC grant EP/S004696/1 and UKRI Future Leaders Fellowship MR/T041609/1.

\section{Preliminaries on the arithmetic of algebraic tori}\label{sec:arithmetic_tori}
Let $T$ be a torus over a number field $k$. As mentioned above, Voskresenski\u{\i}'s exact sequence ties together the Tate--Shafarevich group $\Sha(T)$ and the defect of weak approximation $A(T)$:

\begin{theorem}[Voskresenski\u{\i}]\label{thmvosk}
Let $T$ be a torus defined over a number field $k$ and let $X/k$ be a smooth compactification of $T$. Then there exists an exact sequence
\begin{equation}\label{eq:Voskexact}
    0 \to A(T) \to \operatorname{H}^1(k,\operatorname{Pic}\overline{X})^{\sim} \to \Sha(T) \to 0. 
\end{equation} 

\end{theorem}

\begin{proof}
See \cite[Theorem 6]{Vosk}.
\end{proof}

Voskresenski\u{\i} proved Theorem~\ref{thmvosk} by considering the following exact sequence of Galois modules:
\begin{equation}\label{eq:Voskflasque}
    0\to \widehat{T}\to\operatorname{Div}_{\overline{X}-\overline{T}}\overline{X} \to \operatorname{Pic}\overline{X}\to 0,
\end{equation}
where $\widehat{T}$ denotes the group of characters of $T$. The key point is that \eqref{eq:Voskflasque} is a \emph{flasque resolution} of the Galois module $\widehat{T}$. We explain this concept below (see \cite{coll1} and \cite{coll2} for more details).

Let $G$ be a finite group and let $A$ be a $G$-module. We say that $A$ is a \textit{permutation} module if it has a $\Z$-basis permuted by $G$. We say that $A$ is \textit{flasque} if $\hat{\operatorname{H}}^{-1}(G',A)=0$ for all subgroups $G'$ of $G$. A \textit{flasque resolution} of $A$ is an exact sequence of $G$-modules
$$ 0 \to A \to P \to F \to 0,$$

\noindent where $P$ is a permutation module and $F$ is flasque. We say two $G$-modules $A_1$ and $A_2$ are similar if $A_1 \oplus P_1 \cong A_2 \oplus P_2 $ for permutation modules $P_1,P_2$ and denote the similarity class of $A$ by $[A]$.

Recall that if $T$ is split by a Galois subextension $L/k$ of $\overline{k}/k$, then $\Gal(\overline{k}/L)$ acts trivially on the character group $\widehat{T}=\operatorname{Hom}(\overline{T},\mathbb{G}_{m,\overline{k}})$ and thus $\widehat{T}$ is a $\Gal(L/k)$-module. Implicit in much of our work is the fact that the norm one torus $R^{1}_{K/k} \mathbb{G}_m$ is split by any Galois extension of $k$ containing $K$. The followng result shows that the group $\operatorname{H}^1(k,\operatorname{Pic}\overline{X})$ has a simple cohomological description and can be computed using \emph{any} flasque resolution of $\widehat{T}$.

\begin{theorem}
[Colliot-Thélène and Sansuc]\label{thmcs}
Let $T$ be a torus defined over a number field $k$ and split by a finite Galois extension $L/k$ with $G=\operatorname{Gal}(L/k)$. Let
$$ 0 \to \widehat{T} \to P \to F \to 0
$$

\noindent be a flasque resolution of $\widehat{T}$ and let $X/k$ be a smooth compactification of $T$. Then the similarity class $[F]$ and the group $\operatorname{H}^1(G,F)$ are uniquely determined and
\begin{equation}\label{cs_flasque}
    \operatorname{H}^1(k,\operatorname{Pic}\overline{X}) = \operatorname{H}^1(G,\operatorname{Pic}X_L) = \operatorname{H}^1(G,F).
\end{equation}

\noindent Additionally, 
\begin{equation}\label{cs_sha}
    \operatorname{H}^1(G,F) = \Sha_{\omega}^2(G,\widehat{T}):=\Ker\Bigl({\operatorname{H}}^2(G,\widehat{T}) \xrightarrow[]{\operatorname{Res}} \prod_{g \in G} {\operatorname{H}}^2(\langle g \rangle,\widehat{T})\Bigr).
\end{equation} 

\noindent In the special case where $T=R^1_{L/k} \bbG_m$, we have
\begin{equation}\label{eq:tate_mid}
     \Sha_{\omega}^2(G,\widehat{T})=\operatorname{H}^2(G,\widehat{T})=\operatorname{H}^3(G,\Z).
\end{equation}
\end{theorem}

\begin{proof}
See \cite[Lemme 5 and Proposition 6]{coll1} for the proof of \eqref{cs_flasque}. The isomorphism $\operatorname{H}^1(G,F) = \Sha_{\omega}^2(G,\widehat{T})$ is proved in \cite[Proposition 9.5(ii)]{coll2}. The final assertion for $R^1_{L/k} \bbG_m$ follows from its defining sequence.
\end{proof}

The Tate--Shafarevich group $\Sha(T)$ also has a description in terms of the cohomology of $\widehat{T}$: 

\begin{theorem}[Tate]
\label{Tate}
Let $T$ be a torus defined over a number field $k$ and split by a finite Galois extension $L/k$ with $G=\operatorname{Gal}(L/k)$. Then Poitou--Tate duality gives a canonical isomorphism
\begin{equation}\label{eq:poitou_tate}
    \Sha(T)^{\sim} = \Sha^2(G,\widehat{T}),
\end{equation}
\noindent where $\Sha^2(G,\widehat{T})=\ker\Bigl(\operatorname{H}^2(G,\widehat{T}) \xrightarrow[]{\operatorname{Res}} \prod_{v \in \Omega_k} \operatorname{H}^2(D_v,\widehat{T})\Bigr)$. In the special case where $T=R^1_{L/k} \mathbb{G}_m$, 
\begin{equation}\label{eq:tate_knot}
    \Sha(T)^{\sim} = \Ker\left(\mathrm{H}^3(G,\bbZ)\xrightarrow{\Res}  \prod_{v \in \Omega_k}\mathrm{H}^3(D_v,\bbZ)\right),
\end{equation}

\noindent where $D_v=\Gal(L_v/k_v)$ is the decomposition group at $v$.
\end{theorem}

\begin{proof}
This is the case $i=1$ of \cite[Theorem 6.10]{Platonov}.
For the case $T=R^1_{L/k} \mathbb{G}_m$, see \cite[p.~198]{C-F}. 
\end{proof}

\begin{proposition}\label{prop:dualVosk}Let $T$ be a torus defined over a number field $k$ and split by a finite Galois extension $L/k$ with $G=\operatorname{Gal}(L/k)$.
Then taking duals in Voskresenski\u{\i}'s exact sequence \eqref{eq:Voskexact} yields the exact sequence
\begin{equation}\label{eq:Voskexactdual}
    0\to \Sha^2(G,\widehat{T})\to \Sha^2_\omega(G,\widehat{T})\to A(T)^\sim \to 0,
\end{equation}
where the map $\Sha^2(G,\widehat{T})\to\Sha^2_\omega(G,\widehat{T})$ is the natural inclusion arising from the Chebotarev density theorem. 
\end{proposition}

\begin{proof}
This follows from the proof of \cite[Theorem 6]{Vosk} and isomorphisms \eqref{cs_sha} and \eqref{eq:poitou_tate}.
\end{proof}

Let us return to the case where $T$ is the norm one torus $R^{1}_{K/k} \mathbb{G}_m$ of an extension $K/k$ of number fields. Taking character modules in the defining sequence for $T$ shows that $\widehat{T}$ is isomorphic to the $G$-module $J_{G/H}$, defined as follows:

\begin{definition}[Chevalley module]\label{def:chevalley}
Let $G$ be a finite group and $H$ a subgroup of $G$. The map $\eta : \Z \to \Z[G/H] $ defined by $\eta : 1 \mapsto N_{G/H}=\sum\limits_{gH\in G/H}gH$ produces the exact sequence of $G$-modules
$$ 0 \to \Z \xrightarrow[]{\eta} \Z[G/H] \to J_{G/H}  \to 0 ,$$
    
\noindent where $J_{G/H}=\operatorname{coker}\eta$ is called the \emph{Chevalley module of $G/H$}. 
\end{definition}

 Furthermore, for $T=R^{1}_{K/k} \mathbb{G}_m$ we have
$$\Sha(T) = \mathfrak{K}(K/k) $$
\noindent (see \cite[p.~307]{Platonov}). Hence, Theorem \ref{thmvosk} gives a necessary and sufficient condition for the simultaneous validity of the HNP for $K/k$ and weak approximation for $T$, namely the vanishing of $\operatorname{H}^1(k,\operatorname{Pic}\overline{X})$.

\begin{lemma}\label{lem:degkillsmiddlegp}
Let $K/k$ be a finite extension and let $X$ be a smooth compactification of $T=R^1_{K/k}\bbG_m$. Then $T\times_k K$ is stably rational. Consequently, $\operatorname{H}^1(K, \Pic\overline{X})=0$ and $\operatorname{H}^1(k, \Pic\overline{X})$ is killed by
$[K:k]$.
\end{lemma}

\begin{proof}
Write $T_K=T\times_k K$.
Applying base change to the exact sequence defining $T$ gives 
\begin{equation}
\label{basechange}
1\to T_K\to (R_{K/k}\bbG_m)\times_k K \xrightarrow{N_{K/k}} \bbG_{m,K}\to 1. 
\end{equation}
Let $L/k$ be a Galois extension containing $K$. Let $G=\Gal(L/k)$ and let $H=\Gal(L/K)$.
Taking character groups gives an exact sequence of $H$-modules
\begin{equation}
\label{chargps}
0\to \bbZ \xrightarrow{N_{G/H}}  \bbZ[G/H]\to \widehat{T_K}\to 0
\end{equation}
where $N_{G/H}:1\mapsto \sum\limits_{gH\in G/H}gH$. The map $\sum\limits_{gH\in G/H} a_{gH}\cdot gH\mapsto a_H$
defines a left splitting of \eqref{chargps}. Therefore, \eqref{basechange} splits and consequently
\[T_K\times \bbG_{m,K}\cong (R_{K/k}\bbG_m)\times_k K.\]
Hence, $T_K$ is $K$-stably rational, whereby
$\operatorname{H}^1(K, \Pic\overline{X})=\operatorname{H}^1(H, \Pic X_L)=0.$ Now recall that $\Cor^G_H\circ\Res^G_H$ is multiplication by $[G:H]=[K:k]$ and 
$\Res^G_H: \operatorname{H}^1(G, \Pic X_L)\to \operatorname{H}^1(H, \Pic X_L)=0.$
 This completes the proof that $[K:k]$ kills $\operatorname{H}^1(G, \Pic X_L)=\operatorname{H}^1(k, \Pic\overline{X})$.
\end{proof}
The corollary below is an immediate consequence of Theorem~\ref{thmvosk} and Lemma~\ref{lem:degkillsmiddlegp}.
\begin{corollary}\label{cor:nop}
Let $T=R^1_{K/k}\bbG_m$. Then $A(T)$ and $\mathfrak{K}(K/k)$ are killed by $[K:k]$.
\end{corollary}

\section{Using subextensions and superextensions}\label{sec:upanddown}
Let $k$ be a number field. In order to study the HNP and weak approximation in non-Galois extensions of $k$, it is often useful to be able to deduce information about the knot group of an extension $K/k$ from information about its subextensions or superextensions, the latter meaning extensions of $k$ containing $K$. In this section we collect some results that serve this purpose. 

\begin{lemma}\label{lem:Voskfunct}
Let $\phi:T_1\to T_2$ be a morphism of algebraic tori over $k$, and let $X_1$ and $X_2$ be smooth compactifications of $T_1$ and $T_2$, respectively. Then we obtain a commutative diagram with exact rows as follows, where the vertical arrows are induced by $\phi$:
\[
\xymatrix{0 \ar[r]& A(T_1) \ar[r]\ar[d]& \operatorname{H}^1(k,\operatorname{Pic}\overline{X_1})^{\sim} \ar[r]\ar[d] &\Sha(T_1)\ar[d] \ar[r]& 0\\
0 \ar[r]& A(T_2) \ar[r]& \operatorname{H}^1(k,\operatorname{Pic}\overline{X_2})^{\sim} \ar[r] &\Sha(T_2) \ar[r]& 0.
}
\]
\end{lemma}

\begin{proof}
This follows from Voskresenski\u{\i}'s proof of \cite[Theorem 6]{Vosk}.
\end{proof}

\begin{corollary}\label{cor:CT}
Let $\phi:T_1\to T_2$ be an isogeny of algebraic tori over $k$ with kernel $\mu$.
Let $X_1$ and $X_2$ be smooth compactifications of $T_1$ and $T_2$, respectively. Then for any prime $p$ such that $p\nmid |\mu(\overline{k})|$, we obtain a commutative diagram with exact rows as follows, where the vertical isomorphisms are induced by $\phi$:
\[
\xymatrix{0 \ar[r]& A(T_1)_{(p)} \ar[r]\ar[d]_{\cong}& \operatorname{H}^1(k,\operatorname{Pic}\overline{X_1})^{\sim}_{(p)} \ar[r]\ar[d]_{\cong} &\Sha(T_1)_{(p)}\ar[d]_{\cong} \ar[r]& 0\\
0 \ar[r]& A(T_2)_{(p)} \ar[r]& \operatorname{H}^1(k,\operatorname{Pic}\overline{X_2})^{\sim}_{(p)} \ar[r] &\Sha(T_2)_{(p)} \ar[r]& 0.
}
\]
\end{corollary}

\begin{proof}
Let $\psi:T_2\to T_1$ be the dual isogeny. Then $\psi\circ\phi$ is multiplication by $|\mu(\overline{k})|$
on $T_1$. Now apply Lemma~\ref{lem:Voskfunct}.
\end{proof}

The following theorem is an application of Corollary~\ref{cor:CT} to norm one tori which will be very useful later in this section as well as in Section~\ref{sect_applications}.

\begin{theorem}\label{mid_gp_general_norm}
Let $L/K/k$ be a tower of finite extensions. 
Let $T_0=R^1_{L/k}\bbG_m$, let $T=R^1_{K/k}\bbG_m$ and let $X_0$ and $X$ be smooth compactifications of $T_0$ and $T$, respectively. Then for a prime $p$ with $p\nmid [L:K]$ we obtain a commutative diagram with exact rows as follows, where the vertical isomorphisms are induced by the natural inclusion $j: T\hookrightarrow T_0$:
\[
\xymatrix{0 \ar[r]& A(T)_{(p)} \ar[r]\ar[d]_{\cong}& \operatorname{H}^1(k,\operatorname{Pic}\overline{X})^{\sim}_{(p)} \ar[r]\ar[d]_{\cong} &\Sha(T)_{(p)}\ar[d]_{\cong} \ar[r]& 0\\
0 \ar[r]& A(T_0)_{(p)} \ar[r]& \operatorname{H}^1(k,\operatorname{Pic}\overline{X_0})^{\sim}_{(p)} \ar[r] &\Sha(T_0)_{(p)} \ar[r]& 0.
}
\]
Alternatively, the norm map $N_{L/K}:T_0\twoheadrightarrow T$ can be used to obtain a similar commutative diagram with the direction of the vertical isomorphisms reversed.
\end{theorem}

\begin{proof}
Let $S$ be the kernel of $N_{L/K}:R_{L/k}\bbG_{m}\to R_{K/k}\bbG_m$ and let $i:S\to R_{L/k}\bbG_{m}$ be the inclusion. Then the following diagram with exact rows commutes:
\begin{displaymath}
\xymatrix{
1\ar[r]& S\ar@{=}[d]\ar[r]^{i}&R^1_{L/k}\bbG_m\ar[r]^{N_{L/K}} \ar@{_(->}[d]& R^1_{K/k}\bbG_m\ar[r]\ar@{_(->}[d] & 1\\
1\ar[r]& S\ar[r]^{i}&R_{L/k}\bbG_m\ar[r]^{N_{L/K}} & R_{K/k}\bbG_m\ar[r] & 1.}
\end{displaymath}
Let $d=[L:K]$ and let $[d]$ denote the map $x\mapsto x^d$. The natural inclusion $j:R_{K/k}\bbG_{m}\to R_{L/k}\bbG_m$ satisfies $N_{L/K}\circ j=[d]$. Using $i$ and $j$, we obtain a surjective morphism 
\[S\times R_{K/k}\bbG_m\to R_{L/k}\bbG_m\]
whose kernel $\mu$ is finite for dimension reasons. Moreover, since $N_{L/K}\circ j=[d]$, $\mu$ is killed by $d$. 
Let $Z$, $W$ and $W_0$ be smooth compactifications of $S$, $R_{K/k}\bbG_m$ and $R_{L/k}\bbG_m$, respectively. 
By \cite[Lemma~3]{Vosk}, $\Pic(\overline{Z\times W})=\Pic\overline{Z}\oplus\Pic\overline{W}$. 
Thus, Corollary~\ref{cor:CT} yields
\[\operatorname{H}^1(k, \Pic\overline{Z})_{(p)}\oplus \operatorname{H}^1(k, \Pic\overline{W})_{(p)}\cong \operatorname{H}^1(k, \Pic\overline{W_0})_{(p)}.\]
Furthermore, $R_{K/k}\bbG_m$ and $R_{L/k}\bbG_m$ are $k$-rational so $\operatorname{H}^1(k, \Pic\overline{W})=\operatorname{H}^1(k, \Pic\overline{W_0})=0$ and hence $\operatorname{H}^1(k, \Pic\overline{Z})_{(p)}=0$. Therefore, $\Sha(S)_{(p)}=A(S)_{(p)}=0$ by Theorem~\ref{thmvosk}. Now the result follows from applying Corollary~\ref{cor:CT} to the surjective morphism
\[S\times R^1_{K/k}\bbG_m\to R^1_{L/k}\bbG_m\]
whose finite kernel is killed by $d$. 
\end{proof}

The following special case of Theorem~\ref{mid_gp_general_norm} reduces the calculation of $A(T)$,
$\operatorname{H}^1(k,\operatorname{Pic}\overline{X})$ and $\Sha(T)$ to the case where $K/k$ is the fixed field of a $p$-group.

\begin{corollary}\label{cor:mid_gp_general_intro}
Let $L/K/k$ be a tower of finite extensions with $L/k$ Galois. Let $G=\Gal(L/k)$ and $H=\Gal(L/K)$. For $p$ prime, let $H_p$ denote a Sylow $p$-subgroup of $H$ and let $K_p$ denote its fixed field. Let $X$ and $X_p$ be smooth compactifications of $T=R^1_{K/k}\bbG_m$ and $T_p=R^{1}_{K_p/k} \mathbb{G}_m$, respectively. Then we obtain a commutative diagram with exact rows as follows, where the vertical isomorphisms are induced by the natural inclusion $T\hookrightarrow T_p$:
\[
\xymatrix{0 \ar[r]& A(T)_{(p)} \ar[r]\ar[d]_{\cong}& \operatorname{H}^1(k,\operatorname{Pic}\overline{X})^{\sim}_{(p)} \ar[r]\ar[d]_{\cong} &\Sha(T)_{(p)}\ar[d]_{\cong} \ar[r]& 0\\
0 \ar[r]& A(T_p)_{(p)} \ar[r]& \operatorname{H}^1(k,\operatorname{Pic}\overline{X_p})^{\sim}_{(p)} \ar[r] &\Sha(T_p)_{(p)} \ar[r]& 0.
}
\]
Alternatively, the norm map $N_{K_p/K}:T_p\twoheadrightarrow T$ can be used to obtain a similar commutative diagram with the direction of the vertical isomorphisms reversed.
\end{corollary}

As a consequence of Corollary~\ref{cor:mid_gp_general_intro}, we obtain the following result which deals with the two extremes in terms of the power of $p$ dividing $|H|$. 

\begin{corollary}\label{gp_containment_intro}
Retain the notation of Corollary~\ref{cor:mid_gp_general_intro}.
\begin{enumerate}[label=(\roman{*}),leftmargin=*]
\item \label{no p middle gp} If $p \nmid |H|$, then $\operatorname{H}^1(k,\Pic \overline{X})_{(p)} \cong \operatorname{H}^3(G,\Z)_{(p)}$.
    \item If $H$ contains a Sylow $p$-subgroup of $G$, then $\operatorname{H}^1(k,\Pic \overline{X})_{(p)}=0$.
\end{enumerate}
\end{corollary}

\begin{proof}
\begin{enumerate}[label=(\roman{*})]
 \item\label{gp_containment2} Follows from Theorem~\ref{thmcs} and Corollary~\ref{cor:mid_gp_general_intro}.
  \item\label{gp_containment1} Follows from Lemma~\ref{lem:degkillsmiddlegp}. \qedhere
  \end{enumerate}
\end{proof}

We additionally obtain the following result when $H$ is a \emph{Hall subgroup} of $G$, i.e.~a subgroup such that $\gcd(|H|,[G:H])=1$.

\begin{corollary}\label{hall_middle}
Retain the notation of Corollary~\ref{cor:mid_gp_general_intro}.
If $H$ is a Hall subgroup of $G$, then
\begin{eqnarray*}\operatorname{H}^1(k, \Pic\overline{X})&\cong& \prod\limits_{p \nmid |H|} \operatorname{H}^3(G,\Z)_{(p)},\\
\mathfrak{K}(K/k)&\cong &\prod\limits_{p \nmid |H|}\mathfrak{K}(L/k)_{(p)},
\ \ \textrm{ and}\\
A(T)&\cong& \prod\limits_{p \nmid |H|}A(T_0)_{(p)},
\end{eqnarray*}
where $T=R^1_{K/k}\bbG_m$ and $T_0=R^1_{L/k}\bbG_m$.
\end{corollary}

\begin{proof}
 Follows from Corollaries~\ref{cor:mid_gp_general_intro} and~\ref{gp_containment_intro}, Lemma~\ref{lem:degkillsmiddlegp} and Corollary~\ref{cor:nop}.
\end{proof}

We now drop the assumption that $L/k$ is Galois and return to the more general setting of Theorem~\ref{mid_gp_general_norm}.

\begin{corollary}
\label{cor:goingup1}
Retain the notation of Theorem~\ref{mid_gp_general_norm}. Then:
\begin{enumerate}[label=(\roman{*})]
    \item $A(T)$ is killed by $[L:K]\cdot\exp(A(T_0))$;
    \item $\operatorname{H}^1(k, \Pic\overline{X})$ is killed by $[L:K]\cdot\exp(\operatorname{H}^1(k, \Pic\overline{X_0}))$;
    \item $\Sha(T)$ is killed by $[L:K]\cdot\exp(\Sha(T_0))$.
\end{enumerate}    
  
\end{corollary}

\begin{proof}
We give the proof for $A(T)$ -- the other proofs are analogous. Let $d=[L:K]$, $e=\exp(A(T_0))$ and let $x\in A(T)$. Since $N_{L/K}\circ j=[d]$, we have $x^{de}=N_{L/K}(j(x)^e)=1$, as $j(x)\in A(T_0)$.
\end{proof}

\begin{corollary}\label{cor:goingup2}
Retain the notation of Theorem~\ref{mid_gp_general_norm}. 
\begin{enumerate}[label=(\roman{*}),leftmargin=*]
    \item If $\exp(A(T_0))\cdot [L:K]$ is coprime to $[K:k]$, then weak approximation holds for $K/k$.
    \item If $\exp(\Sha(T_0))\cdot [L:K]$ is coprime to $[K:k]$, then the HNP holds for $K/k$.
\end{enumerate} 
\end{corollary}

\begin{proof}
This follows immediately from Corollaries~\ref{cor:goingup1} and~\ref{cor:nop}.
\end{proof}

The following result is a slight generalization of \cite[Proposition 1]{Gurak}. 

\begin{proposition}
\label{GurakProp1}
Let $L/K/k$ be a tower of finite extensions and let $d=[L:K]$. Then the map $x\mapsto x^d$ induces a group homomorphism
$$\varphi: \mathfrak{K}(K/k)\to \mathfrak{K}(L/k)$$

\noindent with $\Ker\varphi\subset \mathfrak{K}(K/k)[d]$ and $\{x^d\mid x\in \mathfrak{K}(L/k)\}\subset \im\varphi$. In particular, if $|\mathfrak{K}(K/k)|$ is coprime to $d$, then $\varphi$ induces an isomorphism $\mathfrak{K}(K/k) \cong \{x^d\mid x\in \mathfrak{K}(L/k)\}$.
\end{proposition}

\begin{proof}
This follows from the fact that under the identification of $\operatorname{H}^1(k,R^1_{K/k}\bbG_m)$ and $\operatorname{H}^1(k,R^1_{L/k}\bbG_m)$ with $k^*/N_{K/k}K^*$ and $k^*/N_{L/k}L^*$, the maps $j$ and $N_{L/K}$ from Theorem~\ref{mid_gp_general_norm} induce multiplication by $d$ and projection respectively. Alternatively, observe that the proposition follows from the inclusions $N_{L/k}\bbA_L^*\subset  N_{K/k}\bbA_K^*$, $N_{L/k}L^*\subset  N_{K/k}K^*$, $(N_{K/k}\bbA_K^*)^d \subset  N_{L/k}\bbA_L^*$ and $(N_{K/k}K^*)^d\subset N_{L/k}L^*$. If $|\mathfrak{K}(K/k)|$ is coprime to $d$, then $\im \varphi \subset \{x^d\mid x\in \mathfrak{K}(L/k)\}$.
\end{proof}

Next, we establish a generalization of Gurak’s criterion (see \cite[Proposition 2]{Gurak}) for the validity of the HNP in a compositum of two subextensions with coprime degrees.

\begin{proposition}
\label{GurakProp2}
Let $L/k$ be a finite extension with subextensions $K/ k$ and $M/ k$ such that \mbox{$L=KM$}. Let $T=R^1_{L/k}\bbG_m$, $T_1=R^1_{K/k}\bbG_m$ and $T_2=R^1_{M/k}\bbG_m$ and let $X, X_1$ and $X_2$ be their respective smooth compactifications.
Then we obtain a commutative diagram with exact rows as follows, where the vertical homomorphisms are induced by the natural inclusions $T_1\hookrightarrow T$ and $T_2\hookrightarrow T$:
\[
\xymatrix{0 \ar[r]& A(T_1)\oplus A(T_2) \ar[r]\ar[d]& \operatorname{H}^1(k,\operatorname{Pic}\overline{X_1})^{\sim}\oplus \operatorname{H}^1(k,\operatorname{Pic}\overline{X_2})^{\sim} \ar[r]\ar[d] &\Sha(T_1)\oplus \Sha(T_2) \ar[d] \ar[r]& 0\\
0 \ar[r]& A(T) \ar[r]& \operatorname{H}^1(k,\operatorname{Pic}\overline{X})^{\sim} \ar[r] &\Sha(T) \ar[r]& 0.
}
\]

\noindent If $[K:k]$ and $[M:k]$ are coprime, then the vertical maps in the diagram are isomorphisms. 

\end{proposition}

\begin{proof}
The commutative diagram comes from Lemma~\ref{lem:Voskfunct}. If $[K:k]$ and $[M:k]$ are coprime, then any prime number divides at most one of $[L:K]$ and $[L:M]$, whence Lemma~\ref{lem:degkillsmiddlegp} and Theorem~\ref{mid_gp_general_norm} show that the vertical maps in the diagram are isomorphisms. 
\end{proof}

\begin{proposition}\label{GurakProp3}
In the notation of Proposition~\ref{GurakProp2}, the map $\Sha(T_1) \oplus \Sha(T_2) \to \Sha(T)$ induces the following homomorphism on the relevant knot groups
\begin{eqnarray*}
\varphi: \mathfrak{K}(K/k)\times \mathfrak{K}(M/k)\to \mathfrak{K}(L/k)\\
(x,y)\mapsto x^n y^m
\end{eqnarray*}
where $m=[L:M]$ and $n=[L:K]$. Moreover, if $a=\exp(\mathfrak{K}(K/k))$, $b=\exp(\mathfrak{K}(M/k))$, and $h=\gcd(m,n)$, then 
$\varphi$ satisfies $\Ker\varphi\subset \mathfrak{K}(K/k)[bn]\times \mathfrak{K}(M/k)[am]$ and $\{z^h\mid z\in \mathfrak{K}(L/k)\}\subset \im\varphi$.

\end{proposition}

\begin{proof}
This follow from the argument in the proof of Proposition~\ref{GurakProp1}.
\end{proof}

We end this section by proving a version of \cite[Theorem 1]{Gurak} for weak approximation in nilpotent
Galois extensions. We require the following description of the defect of weak approximation:

\begin{proposition}\label{Tate_WA}
Let $T$ be a torus defined over a number field $k$ and split by a finite Galois extension $L/k$ with $G=\operatorname{Gal}(L/k)$.
Then 
\begin{equation}\label{eq:WA_general}
    A(T)^{\sim}=\im\left(\Sha^2_\omega(G,\widehat{T})\xrightarrow{\Res}  \prod_{v \in \Omega_k}\mathrm{H}^2(D_v,\widehat{T})\right)
\end{equation}
    \noindent where $D_v=\Gal(L_v/k_v)$ is the decomposition group at $v$. If $T=R^1_{L/k}\mathbb{G}_m$ then
    \begin{equation}\label{eq:WA_galois}
    A(T)^{\sim}=\im\left(\mathrm{H}^3(G,\bbZ)\xrightarrow{\Res}  \prod_{v \in \Omega_k}\mathrm{H}^3(D_v,\bbZ)\right).
\end{equation}
\end{proposition}

\begin{proof}
The equality in \eqref{eq:WA_general} follows from Proposition~\ref{prop:dualVosk}. Then \eqref{eq:WA_galois} follows from the isomorphism~\eqref{eq:tate_mid} of Theorem~\ref{thmcs} and the analogous result that $\mathrm{H}^2(D_v,\widehat{T})=\mathrm{H}^3(D_v,\bbZ)$ in this setting.
\end{proof}

We make use of the following weak approximation version of \cite[Lemma~2.3]{GurakcyclicSylow}:

\begin{lemma}\label{lem:wa_composing}
Let $K/k$ and $M/k$ be finite subextensions of $L/k$ such that $[K:k]$ and $[M:k]$ are coprime. If weak approximation holds for $R^1_{KM/M}\mathbb{G}_m$, then it holds for $R^1_{K/k}\mathbb{G}_m$. Under the additional assumption that $K/k$ is Galois, weak approximation for $R^1_{K/k}\mathbb{G}_m$ implies weak approximation for $R^1_{KM/M}\mathbb{G}_m$.
\end{lemma}

\begin{proof} Let $T=R^1_{K/k}\mathbb{G}_m$, $T_M=T\times_k M$ and $T_K=T\times_k K$. Suppose first that weak approximation holds for $R^1_{KM/M}\bbG_m=T_M$. By Lemma~\ref{lem:degkillsmiddlegp} and Theorem~\ref{thmvosk}, weak approximation holds for $T_K$.
To complete the proof, observe that weak approximation for $T_K$ and $T_M$ implies weak approximation for $R_{K/k}T_K$ and $R_{M/k}T_M$.
Since $[K:k]$ and $[M:k]$ are coprime, the surjective morphism of algebraic groups
\begin{eqnarray*}
 R_{K/k}T_K  \times R_{M/k}T_M  & \to & T\\
 (x,y) & \to &  N_{K/k}(x)N_{M/k}(y)
\end{eqnarray*}
has a section. Therefore, weak approximation for $T$ follows from weak approximation for $R_{K/k}T_K$ and $R_{M/k}T_M$.

Now suppose that $K/k$ is Galois and that weak approximation holds for $R^1_{K/k}\mathbb{G}_m$. Then $KM/M$ is Galois with Galois group isomorphic to $\Gal(K/k)$. Let $w$ be a place of $M$ and let $v$ be the place of $k$ lying below $w$.
The various restriction maps give a commutative diagram
\[
\xymatrix{
\operatorname{H}^3(\Gal(K/k),\Z)\ar[d]^{ \Res_v}\ar[r]^{\cong} & \operatorname{H}^3(\Gal(KM/M),\Z)\ar[d]^{ \Res_w}\\
\operatorname{H}^3(D_v,\Z)\ar[r] & \operatorname{H}^3(D_w,\Z).
}
\]

\noindent Since weak approximation holds for $R^1_{K/k}\mathbb{G}_m$, isomorphism~\eqref{eq:WA_galois} of Proposition~\ref{Tate_WA} shows that $\Res_v$ is trivial, and hence $\Res_w$ is also trivial. As $w$ was arbitrary, weak approximation for $R^1_{KM/M}\mathbb{G}_m$ follows from~\eqref{eq:WA_galois}.
\end{proof}

\begin{remark}
The hypothesis that $K/k$ is Galois in the second implication of Lemma~\ref{lem:wa_composing} is necessary. To see this, consider a Galois extension $L/k$ with Galois group $G=C_3 \times S_3$ and with a decomposition group $D_v$ containing the
Sylow $3$-subgroup of $G$ for some place $v$ of $k$ (such an extension always exists, see Section~\ref{sec:eg}). 
Let $K/k$ and $M/k$ be subextensions of $L/k$ of degree $9$ and $2$, respectively. One can verify that the invariant $\operatorname{H}^1(k,\Pic \overline{X})$ vanishes for $K/k$ and thus weak approximation holds for $R^1_{K/k}\mathbb{G}_m$ by Theorem~\ref{thmvosk}. On the other hand, $KM/M=L/M$ is Galois with Galois group $C_3 \times C_3$ and decomposition group $C_3\times C_3$ for a prime of $M$ above $v$. 
It follows that weak approximation fails for $R^1_{KM/M}\mathbb{G}_m$ by isomorphism~\eqref{eq:WA_galois} of Proposition~\ref{Tate_WA}. See \cite{Liang} for some other examples of varieties over number fields that satisfy weak approximation over the base field but not over a quadratic extension.
\end{remark}

We also require the following well-known fact:

\begin{proposition}
\label{prop:noptorsion}
Let $G$ be a finite group and $G_p$ a Sylow $p$-subgroup of $G$. For any $G$-module $A$ and any $n\in\bbZ_{>0}$, the restriction map
$$\Res^G_{G_p}: \mathrm{H}^n(G ,A) \rightarrow \mathrm{H}^n(G_p ,A)$$ 
maps $\mathrm{H}^n(G ,A)_{(p)}$ injectively into $ \mathrm{H}^n(G_p ,A)$. 
\end{proposition}

\begin{proof}
See, for example, \cite[Theorem III.10.3]{Brown}.
\end{proof}

\begin{proposition}\label{prop:Gurak_WA_kp_Lp}
Let $L/k$ be a Galois extension such that $G=\Gal(L/k)$ is nilpotent. For every prime $p$, let $G_p$ be a Sylow $p$-subgroup of $G$. Let $k_p$ and $L_p$ be the fixed fields of the subgroups $G_p$ and $\prod\limits_{q \neq p} G_q$, respectively. The following are equivalent:

\begin{enumerate}[label=(\roman{*})]
    \item\label{Gurak_WA_1} Weak approximation holds for $R^1_{L/k}\mathbb{G}_m$.
    \item\label{Gurak_WA_2} Weak approximation holds for each $R^1_{L_p/k}\mathbb{G}_m$.
    \item\label{Gurak_WA_3} Weak approximation holds for each $R^1_{L/k_p}\mathbb{G}_m$.
\end{enumerate}
\end{proposition}

\begin{proof}
\ref{Gurak_WA_1} $\implies$ \ref{Gurak_WA_2}: Follows from Corollary~\ref{cor:goingup2}.

\ref{Gurak_WA_2} $\implies$ \ref{Gurak_WA_3}: 
Follows from Lemma~\ref{lem:wa_composing}.

\ref{Gurak_WA_3} $\implies$ \ref{Gurak_WA_1}: We prove $A(R^1_{L/k}\mathbb{G}_m)_{(p)}=0$ for every prime $p$. Let $v$ be a place of $k$ and let $w$ be a place of $k_p$ above $v$.
The various restriction maps give a commutative diagram
\begin{equation*}\label{diag:cor}
\xymatrix{\operatorname{H}^3(G,\Z)_{(p)} \ar[r]^{\Res_1} \ar[d]^{\Res_4} &  \operatorname{H}^3(D_v,\Z)_{(p)} \ar[d]^{\Res_2}\\
\operatorname{H}^3(G_p,\Z) \ar[r]^{\Res_3} &  \operatorname{H}^3(D_w,\Z)}
\end{equation*}
\noindent As weak approximation holds for $R^1_{L/k_p}\mathbb{G}_m$, isomorphism~\eqref{eq:WA_galois} of Proposition \ref{Tate_WA} yields $\im \Res_3=0$. Furthermore, Proposition \ref{prop:noptorsion} shows that $\Res_2$ is injective. It follows that $\im \Res_1 = 0$ and, since $v$ was arbitrary, we conclude that $A(R^1_{L/k}\mathbb{G}_m)_{(p)}=0$ by~\eqref{eq:WA_galois}. \end{proof}

\begin{remark}
We note that the implication \ref{Gurak_WA_3} $\implies$ \ref{Gurak_WA_1} in Proposition~\ref{prop:Gurak_WA_kp_Lp} does not require the hypothesis that $G$ is nilpotent. This is analogous to the corresponding result for the HNP -- see Gurak's remarks preceding \cite[Theorem 2]{Gurak}. 
\end{remark}

\section{The first obstruction to the Hasse norm principle}\label{sec:1obs} 
In this section, we give some background concerning the first obstruction to the Hasse norm principle and then go on to prove Theorem~\ref{thm:unramified_description_intro} and Corollary~\ref{an_sn_p_part_intro}.
Throughout the section, we fix a tower of number fields $L/K/k$ such that $L/k$ is Galois. Let $X$ and $X_0$ be smooth compactifications of the tori $R^1_{K/k} \bbG_m$ and $R^1_{L/k} \bbG_m$, respectively. Applying Lemma~\ref{lem:Voskfunct} to the norm map \mbox{$N_{L/K}:R^1_{L/k} \bbG_m \to R^1_{K/k} \bbG_m$} gives a commutative diagram with exact rows as follows, where the vertical arrows are induced by $N_{L/K}$:
\begin{equation}\label{eq:first_obs}
\xymatrix{0 \ar[r]& A(R^1_{L/k} \bbG_m ) \ar[r]\ar[d]& \operatorname{H}^1(k,\operatorname{Pic}\overline{X_0})^{\sim} \ar[r]\ar[d]^{f_{L/K}} &\Sha(R^1_{L/k} \bbG_m )\ar[d]^{g_{L/K}} \ar[r]& 0\\
0 \ar[r]& A(R^1_{K/k} \bbG_m) \ar[r]& \operatorname{H}^1(k,\operatorname{Pic}\overline{X})^{\sim} \ar[r] &\Sha(R^1_{K/k} \bbG_m) \ar[r]& 0.
}
\end{equation}

\begin{definition}\label{def:cokers}
In the notation of diagram \eqref{eq:first_obs}, we define
\begin{enumerate}
    \item $\mathfrak{F}(L/K/k):=\Coker(g_{L/K})=(k^*\cap N_{K/k}\bbA_K^*)/(k^*\cap N_{L/k}\bbA_L^*)N_{K/k}K^*$, called the \emph{first obstruction to the HNP for $K/k$ corresponding to the tower $L/K/k$}, see \cite[Definition 1]{DP};
    \item $\mathfrak{F}_{nr}(L/K/k):=\Coker(f_{L/K})$, called the \textit{unramified cover of $\mathfrak{F}(L/K/k)$}.
\end{enumerate}
\end{definition}

Clearly the knot group $\mathfrak{K}(K/k)$ (which is sometimes called the total obstruction to the HNP) surjects onto $\mathfrak{F}(L/K/k)$ and $\mathfrak{F}(L/K/k)$ equals $\mathfrak{K}(K/k)$ if the HNP holds for $L/k$. In \cite{DP}, Drakokhrust and Platonov give another very useful sufficient criterion for this equality to hold, as follows:

\begin{theorem}{\cite[Theorem 3]{DP}}\label{thm3DP}
Set $G=\Gal(L/k),H=\Gal(L/K)$. Let $G_1, \dots , G_r$ be subgroups of $G$ and let $H_1, \dots , H_r$ be subgroups of $H$ such that $H_i\subset H\cap G_i$ for each $i$. Let $K_i=L^{H_i}$ and $k_i=L^{G_i}$. Suppose that the HNP holds for the extensions $K_i/k_i$ and that the map
$$\bigoplus_{i=1}^r\Cor^{G}_{G_i}: \bigoplus_{i=1}^r \hat{\operatorname{H}}^{-3}(G_i, \bbZ) \to \hat{\operatorname{H}}^{-3}(G, \bbZ)$$
is surjective.
Then $\mathfrak{F}(L/K/k)=\mathfrak{K}(K/k).$
\end{theorem}

In order to compute $\mathfrak{F}(L/K/k)$, Drakokhrust and Platonov give some explicit results relating this object to the local and global Galois groups of the tower $L/K/k$. We present their results here in a slightly more general setting. Let $G$ be a finite group, let $H \leq G$, and let $S$ be a set of subgroups of $G$. Consider the following commutative diagram:
\begin{equation}\label{1stob_general}
\xymatrix{H/[H,H]\ar[r]^{\psi_1} & G/[G,G]\\
\bigoplus\limits_{D \in S}\left( \bigoplus\limits_{H x_i D \in H \backslash G / D}{H_i/[H_i, H_i]} \right) \ar[r]^{\ \ \ \ \ \ \ \ \ \ \ \psi_2} \ar[u]^{\varphi_1}&\bigoplus\limits_{D \in S}{D/[D, D]}\ar[u]_{\varphi_2}
}
\end{equation}

\noindent where the $x_i$'s are a set of representatives of the $H$--$D$ double cosets of $G$, the sum over $D$ is a sum over all subgroups in $S$, and $H_i := H \cap x_i D x_i^{-1}$. 
The maps $\psi_1, \varphi_1$ and $\varphi_2$ are induced by the natural inclusions $H\hookrightarrow G$, $H_i\hookrightarrow H$ and $D\hookrightarrow G$, respectively. If $h\in H_i$, then 
\[\psi_2(h[H_i,H_i])=x_i^{-1}hx_i[D,D] \in D/[D,D].\]

\noindent Given a subgroup $D \in S$, we denote by $\psi_2^D$ the restriction of the map $\psi_2$ in diagram \eqref{1stob_general} to the subgroup $\bigoplus\limits_{H x_i D \in H \backslash G / D}{H_i/[H_i, H_i]}$.

\begin{lemma}\label{lem2DP}
In diagram \eqref{1stob_general}, $\varphi_1(\Ker\psi_2^{D}) \subset \varphi_1(\Ker\psi_2^{D'})$ whenever $D\subset D'$. 
\end{lemma}

\begin{proof}
The proof follows in the same manner as the proof of \cite[Lemma 2]{DP}.
\end{proof}

\begin{lemma}\cite[Lemma 1]{DP}\label{lem1DP} Set $G=\Gal(L/k)$ and $H=\Gal(L/K)$.
Given a place $v$ of $k$, the set of places $w$ of $K$ above $v$ is in one-to-one correspondence with the set of double cosets in the decomposition $G = \bigcup\limits_{i=1}^{r_v} H x_i D_v$. If $w$ corresponds to $H x_{i} D_v$, then the decomposition group $H_w$ of the extension $L/K$ at ${w}$ equals $H \cap x_{i} D_v x_{i}^{-1}$.
\end{lemma}

Set $G=\Gal(L/k)$, $H=\Gal(L/K)$ and $S=\{ D_v \mid v \in \Omega_k \}$. Lemma~\ref{lem1DP} shows that, with these choices, diagram \eqref{1stob_general} takes the form 
\begin{equation}\label{1stob}
\xymatrix{H/[H,H]\ar[r]^{\psi_1} & G/[G,G]\\
\bigoplus\limits_{v \in \Omega_k}\left(\bigoplus\limits_{w | v}{H_w/[H_w, H_w]}\right)\ar[r]^{\ \ \ \ \psi_2} \ar[u]^{\varphi_1}&\bigoplus\limits_{v \in \Omega_k}{D_v/[D_v, D_v]}\ar[u]_{\varphi_2}
}
\end{equation}

\noindent where the sum over $w\mid v$ is a sum over all places $w$ of $K$ above $v$ and $H_w$ is the decomposition group of $L/K$ at $w$.

\begin{theorem}{\cite[Theorem 1]{DP}}\label{thm1DP}
With the notation of diagram \eqref{1stob}, there is a canonical isomorphism
$$\mathfrak{F}(L/K/k) = \Ker\psi_1/\varphi_1(\Ker\psi_2).$$
\end{theorem}

We write $\psi_2^{nr}$ for the restriction of the map $\psi_2$ 
to the subgroup \[\bigoplus_{v \textrm{ unramified in } L/k}\Bigl(\bigoplus_{w\mid v}{H_w/[H_w, H_w]}\Bigr)\] 
\noindent and define $\psi_2^{r}$ similarly using the ramified places. 

\begin{lemma}\label{lem:cyclic=unram}
Set $G=\Gal(L/k)$ and $H=\Gal(L/K)$.
Let $C$ be the set of all cyclic subgroups of $G$ and let $\varphi_1^C$ and $\psi_2^C$ denote the relevant maps in diagram~\eqref{1stob_general} with $S=C$. Then 
\[\varphi_1(\Ker\psi_2^{nr})=\varphi_1^C(\Ker\psi_2^C)\]
where the maps in the expression on the left are the ones in diagram~\eqref{1stob}.
\end{lemma}

\begin{proof}
This follows from the Chebotarev density theorem and Lemma~\ref{lem2DP}.
\end{proof}

\begin{definition}\label{def:focal}
Let $H$ be a subgroup of a finite group $G$. The \emph{focal subgroup of $H$ in $G$} is 
\begin{eqnarray*}
\Phi^G(H)&=&\langle h_1^{-1}h_2 \mid h_1,h_2 \in H \text{ and } h_2 \text{ is } G \text{-conjugate to } h_1 \rangle\\
&=&\langle [h,x] \mid h\in H\cap xHx^{-1}, x\in  G\rangle
 \trianglelefteq H.
\end{eqnarray*}
\end{definition}

\begin{theorem}{\cite[Theorem 2]{DP}}\label{thm2DP}
In the notation of diagram \eqref{1stob}, we have
\[\varphi_1(\Ker\psi_2^{nr})=\Phi^G(H)/[H,H].\]
\end{theorem}

Theorem \ref{thm2DP} is very useful -- quite often one can show that $\Phi^G(H)=H\cap [G,G]$ and hence the first obstruction $\mathfrak{F}(L/K/k)$ is trivial. In fact, since $[N_G(H),H]\subset \Phi^G(H)$, if one can show that $[N_G(H),H]=H\cap[G,G]$, then $\mathfrak{F}(L/K/k)=1$. This criterion generalizes \cite[Theorem 3]{Gurak}.

\begin{remark}\label{1obs_alg}
The group $\Ker \psi_1 / \varphi_1(\Ker \psi_2)$ featured in Theorem~\ref{thm1DP} can be computed in finite time. Indeed, $\Ker \psi_1$ is given in terms of the relevant Galois groups, and by \cite[p.~307]{DP} we have
\begin{equation}\label{phi1_prod}
\varphi_1(\Ker\psi_2)=\varphi_1(\Ker\psi_2^{nr})\varphi_1(\Ker\psi_2^{r}).
\end{equation}
Hence, Theorem~\ref{thm2DP} and the fact that only finitely many places of $k$ ramify in $L/k$ show that $\varphi_1(\Ker\psi_2)$ can be obtained by a finite computation. We combined these facts to assemble a function \code{1obs(G,H,l)} in GAP \cite{gap} that, given the groups $G=\Gal(L/k)$, $H=\Gal(L/K)$ and the list $l$ 
of 
decomposition groups $D_v$ at the ramified places $v$, returns the group $\Ker \psi_1 / \varphi_1(\Ker \psi_2)$ isomorphic to the first obstruction $\mathfrak{F}(L/K/k)$. The code for this function is available at \cite{macedo_code}.
\end{remark}
Our next task is to prove Theorem~\ref{thm:unramified_description_intro}, which gives a purely group-theoretic description of $\mathfrak{F}_{nr}(L/K/k)$. First, recall the definition of the group $\mathfrak{F}(G,H)$:

\begin{definition}\label{def:unramified_1st}
Let $G$ be a finite group and let $H \leq G$. We define the group $\mathfrak{F}(G,H)$ as \[\mathfrak{F}(G,H)=(H \cap [G,G])  / \Phi^G(H).\] 
\end{definition}

Returning to the situation of a tower of number fields $L/K/k$ with $L/k$ Galois, $G=\Gal(L/k)$ and $H=\Gal(L/K)$ and letting $\psi_1, \varphi_1^C$ and $\psi_2^C$ denote the relevant maps in diagram~\eqref{1stob_general} with $S=C$, the set of all cyclic subgroups of $G$, we have 
\begin{equation}\label{f(g,h)_isom}
    \mathfrak{F}(G,H)= \Ker\psi_1/\varphi_1^C(\Ker\psi_2^{C}).
\end{equation}
We now prove the following strengthening of Theorem~\ref{thm:unramified_description_intro}:

\begin{theorem}\label{thm:unramified_description}
There is a canonical isomorphism
$\mathfrak{F}_{nr}(L/K/k)=\mathfrak{F}(G,H)$ under which the natural surjection $\mathfrak{F}_{nr}(L/K/k)\twoheadrightarrow \mathfrak{F}(L/K/k)$ coincides with the natural surjection $\mathfrak{F}(G,H)\twoheadrightarrow \mathfrak{F}(L/K/k)$ induced by Theorem~\ref{thm1DP}.
\end{theorem}
\begin{proof}
The norm map $N_{L/K}$ induces a commutative diagram of $k$-tori with exact lines:
\begin{equation}\label{diag:tori}
\begin{tikzcd}
1 \arrow[r] & R^1_{L/k} \bbG_m \arrow[r] \arrow[d, "N_{L/K}", no head] \arrow[d] & R_{L/k} \bbG_m \arrow[r] \arrow[d, "N_{L/K}"] &  \bbG_m \arrow[r] \arrow[d, phantom] \arrow[d, "="] & 1 \\
1 \arrow[r] & R^1_{K/k} \bbG_m \arrow[r]                                         & R_{K/k} \bbG_m \arrow[r]                      &  \bbG_m \arrow[r]                                   & 1
\end{tikzcd}
\end{equation}

\noindent Taking character groups in \eqref{diag:tori} and then taking $G$-cohomology gives the following commutative diagram of abelian groups with exact lines:
\begin{equation}\label{diag:tori_gps}
\begin{tikzcd}
\operatorname{H}^2(G,\Z) \arrow[r, "\theta_1"] \arrow[d, "="]     & \operatorname{H}^2(G,\Z[G/H]) \arrow[r, "\theta_2"] \arrow[d] & \operatorname{H}^2(G,\widehat{T}) \arrow[r, "\theta_3"] \arrow[d, "f_{L/K}^*"] & \operatorname{H}^3(G,\Z) \arrow[d, "="] \\
\operatorname{H}^2(G,\Z)  \arrow[r] & \operatorname{H}^2(G,\Z[G])=0 \arrow[r]                                   & \operatorname{H}^2(G,\widehat{T_0})  \arrow[r]                                    &\operatorname{H}^3(G,\Z)             
\end{tikzcd}
\end{equation}

\noindent Note that, by Theorem~\ref{thmcs}, the group $\mathfrak{F}_{nr}(L/K/k)=\Coker\left(f_{L/K}:\operatorname{H}^1(k,\Pic \overline{X_0})^{\sim} \to \operatorname{H}^1(k,\Pic \overline{X})^{\sim} \right)$ is dual to $\Ker\left(f_{L/K}^*|_{\Sha^2_{\omega}(G,\widehat{T})} : \Sha^2_{\omega}(G,\widehat{T}) \to \Sha^2_{\omega}(G,\widehat{T_0})\right)$. As the first line of diagram~\eqref{diag:tori_gps} is exact, we have $$\Ker\left(f_{L/K}^*|_{\Sha^2_{\omega}(G,\widehat{T})}\right) = \im\theta_2 \cap \Sha^2_{\omega}(G,\widehat{T}).$$
Furthermore, taking character groups in the second line of \eqref{diag:tori} and then taking both $G$-cohomology and $\langle g\rangle$-cohomology, we obtain the following commutative diagram with exact lines
\begin{equation}\label{diag:tori_gps2}
    \begin{tikzcd}
\operatorname{H}^2(G,\Z) \arrow[r, "\theta_1"] \arrow[d]     & \operatorname{H}^2(G,\Z[G/H]) \arrow[r, "\theta_2"] \arrow[d, "\theta_4"] & \operatorname{H}^2(G,\widehat{T}) \arrow[d] \\
\prod\limits_{g \in G}\operatorname{H}^2(\langle g \rangle,\Z) \arrow[r, "\theta_5"]     &\prod\limits_{g \in G} \operatorname{H}^2(\langle g \rangle,\Z[G/H]) \arrow[r]  & \prod\limits_{g \in G} \operatorname{H}^2(\langle g \rangle,\widehat{T})      
\end{tikzcd}
\end{equation}

\noindent and a straightforward diagram chase shows that $\theta_2$ induces an isomorphism $$ \theta_4^{-1}(\im\theta_5)/\im\theta_1\cong \im\theta_2 \cap \Sha^2_{\omega}(G,\widehat{T}).$$

In \cite[Theorem 6.12]{Platonov} and pages leading to it, the authors show that the first square in diagram \eqref{diag:tori_gps2} is dual to diagram \eqref{1stob_general} with $S=C=\{\textrm{cyclic subgroups of } G\}$, reproduced below:
\begin{equation}\label{1stob2}
\xymatrix{H/[H,H]\ar[r]^{\psi_1} & G/[G,G]\\
\bigoplus\limits_{g \in G}\bigl(\bigoplus\limits_{H x_i \langle g \rangle \in H \backslash G / \langle g \rangle }{\langle x_i g x_i^{-1}\rangle \cap H}\bigr)\ar[r]^{\ \  \ \ \ \ \ \ \ \ \ \ \ \ \ \ \psi_2^C} \ar[u]^{\varphi_1^C}&\bigoplus\limits_{g \in G}{\langle g \rangle}\ar[u]_{}
}
\end{equation}
In particular, $\theta_4^{-1}(\im\theta_5)/\im\theta_1$ is dual to $\Ker \psi_1 / \varphi_1^C (\Ker \psi^C_2)$ and the existence of a canonical isomorphism
$\mathfrak{F}_{nr}(L/K/k)=\mathfrak{F}(G,H)$ follows from \eqref{f(g,h)_isom}. Theorem~\ref{thm1DP} can be proved in an analogous way by considering a version of diagram \eqref{diag:tori_gps2} with all decomposition groups in place of all cyclic subgroups of $G$ and recalling from Theorem~\ref{Tate} that $\Sha(T)$ is dual to $\Sha^2(G,\widehat{T})$. Proposition~\ref{prop:dualVosk} now yields the desired compatibility.
\end{proof}

\begin{proof}[Proof of Corollary~\ref{an_sn_p_part_intro}]
This is a direct consequence of diagram~\eqref{eq:first_obs} and Theorems~\ref{thmcs} and ~\ref{thm:unramified_description}.\end{proof}

\begin{corollary}
If $H$ is a Hall subgroup of $G$, then $\mathfrak{F}_{nr}(L/K/k)=\mathfrak{F}(L/K/k)=1$.
\end{corollary}

\begin{proof}
The focal subgroup theorem \cite{Higman} shows that for a Hall subgroup $H$ of $G$, we have $\mathfrak{F}(G,H)=1$. The result therefore follows from Theorem~\ref{thm:unramified_description_intro} and the surjection $\mathfrak{F}_{nr}(L/K/k) \twoheadrightarrow \mathfrak{F}(L/K/k)$.
\end{proof}

\section{Generalized representation groups}\label{sec:gen_rep}

Theorem~\ref{mid_gp_generalized} below gives an explicit description of the birational invariant $\operatorname{H}^1(k,\Pic \overline{X})$ in terms of generalized representation groups, which we now define:

\begin{definition}\label{gen_rep_gp_defn}
Let $G$ be a finite group. A finite group $\overline{G}$ is called a \emph{generalized representation group} of $G$ if there exists a central extension 
\begin{equation}\label{eq:genrep}
    1 \to K \to \overline{G} \xrightarrow[]{\lambda} G \to 1,
\end{equation}
\noindent such that the transgression map 
$\operatorname{Tr}_G : \hat{\operatorname{H}}^{1}(K,\Q/\Z) \to \hat{\operatorname{H}}^{2}(G,\Q/\Z)$ in the 
inflation-restriction exact sequence is surjective. We call $K$ the base normal subgroup of $\overline{G}$.
\end{definition}

\begin{remark}
 Surjectivity of the transgression map $\operatorname{Tr}_G$ in Definition~\ref{gen_rep_gp_defn} is equivalent to injectivity of the dual map $\operatorname{Tr}_G^*$ in the exact sequence $\hat{\operatorname{H}}^{-3}(G,\Z) \xrightarrow{\operatorname{Tr}_G^*} \hat{\operatorname{H}}^{-2}(K,\Z)\rightarrow \hat{\operatorname{H}}^{-2}(\overline{G},\Z)$, where the second map is induced by the inclusion $K\subset \overline{G}$. Hence, a central extension as in \eqref{eq:genrep} gives a generalized representation group if and only if $\operatorname{Tr}_G^*$ gives an isomorphism $\hat{\operatorname{H}}^{-3}(G,\Z) \cong K \cap [\overline{G},\overline{G}].$ 
\end{remark}

Let $L/k$ be a Galois extension of number fields with Galois group $G$ and let $\overline{G}$ be a generalized representation group of $G$ with base normal subgroup $K$. Then the middle group in Voskresenski\u{\i}'s exact sequence \eqref{eq:Voskexact} for $R^1_{L/k}\bbG_m$ is $\hat{\operatorname{H}}^{-3}(G,\Z)\cong K\cap [\overline{G},\overline{G}]=\mathfrak{F}(\overline{G},K)$. Theorem~\ref{mid_gp_generalized} below shows that this is a special case of a more general phenomenon.

\begin{theorem}[Drakokhrust]\label{mid_gp_generalized}Let $L/K/k$ be a tower of number fields with $L/k$ Galois.
Let $X/k$ be a smooth compactification of the norm one torus $R^1_{K/k}\mathbb{G}_m$. Let $G=\Gal(L/k)$ and $H=\Gal(L/K)$. Let $\overline{G}$ be a generalized representation group of $G$ with projection map $\lambda$ and for any subgroup $B \leq G$ let $\overline{B}=\lambda^{-1}(B)$. Then there is a canonical isomorphism $$\operatorname{H}^1(k,\Pic \overline{X})^{\sim} = \mathfrak{F}(\overline{G},\overline{H}).$$
\end{theorem}

\begin{proof}
For any $v \in \Omega_{k}$, define
\[S_v=\begin{cases}
\lambda^{-1}(D_v)\textrm{ if $v$ is ramified in $L/k$;}\\
 \textrm{a cyclic subgroup of $\lambda^{-1}(D_v)$ with $\lambda(S_v)=D_v$ otherwise.}
\end{cases}\]

\noindent Consider the version of diagram \eqref{1stob_general} with respect to the groups $\overline{G}$, $\overline{H}$ and $S=\{S_v \mid v \in \Omega_k\}$. In this setting, Drakokhrust shows in \cite[Theorem 2]{Drak} that $$\operatorname{H}^1(k,\Pic \overline{X})^{\sim} = \Ker \psi_1/ \varphi_1(\Ker \psi^{nr}_2),$$
where $\psi^{nr}_2$ denotes the restriction of $\psi_2$ to the subgroup \[\bigoplus_{v \textrm{ unramified in } L/k}\Bigl(\bigoplus_{i=1}^{r_v}{\overline{H}\cap x_i S_vx_i^{-1}}\Bigr)\] 
and the $x_i$'s are a set of representatives for the double coset decomposition $\overline{G}=\bigcup\limits_{i=1}^{r_v}\overline{H}x_i S_v$.

By the Chebotarev density theorem we can choose the subgroups ${S}_v$ for $v$ unramified in such  a way that every cyclic subgroup of $\overline{G}$ is in $S$. For this choice, we obtain
$$\Ker \psi_1/ \varphi_1(\Ker \psi^{nr}_2) = \mathfrak{F}(\overline{G},\overline{H}).$$

\noindent Indeed, we clearly have $\Ker \psi_1 = (\overline{H} \cap [\overline{G},\overline{G}])/[\overline{H},\overline{H}]$ and the equality $\varphi_1(\Ker \psi^{nr}_2) = \Phi^{\overline{G}}(\overline{H}) / [\overline{H},\overline{H}]$ follows from Lemma~\ref{lem:cyclic=unram} and an argument similar to the proof of \cite[Theorem 2]{DP}.\end{proof}

The following lemma will be used alongside Theorem~\ref{mid_gp_generalized} in the proof of Theorem~\ref{main0}, see Proposition~\ref{1obs_even} below.
\begin{lemma}\label{summary_objects} We have
 $\mathfrak{F}(\overline{G},\overline{H})\cong \mathfrak{F}(G,H)$ if and only if $\Ker \lambda \cap [\overline{G},\overline{G}] \subset \Phi^{\overline{G}}(\overline{H})$, where the notation is as in Theorem~\ref{mid_gp_generalized}.
\end{lemma}

\begin{proof}
Easy exercise.
\end{proof}

The next lemma enables us to employ generalized representation groups to calculate knot groups using the isomorphism \eqref{eq:tate_knot} of Theorem~\ref{Tate} (via duality) or Theorem~\ref{thm3DP}.

\begin{lemma}{\cite[Lemma 4]{DP}}
\label{kernel_res}
Let $G$ be a finite group, $H$ a subgroup of $G$ and $\overline{G}$ a generalized representation group of $G$ with projection map $\lambda$ and base normal subgroup $K$. Then \[\im \left(\Cor : \hat{\operatorname{H}}^{-3}(H,\Z) \to \hat{\operatorname{H}}^{-3}(G,\Z) \right) \cong K \cap [\lambda^{-1}(H),\lambda^{-1}(H)].\]
\end{lemma}

It is well known that every finite group has a generalized representation group 
(see \cite[Theorem~2.1.4]{kar}). The following result of Schur, which gives presentations of generalized representation groups of $A_n$ and $S_n$, will be used in Section~\ref{sect_applications} when investigating the Hasse norm principle for $A_n$ and $S_n$ extensions.
\begin{proposition}\label{pres}
Let $n \geq 4$ and let $U$ be the group with generators $z,\overline{t_1},\dots, \overline{t_{n-1}}$ and relations
\begin{enumerate}[label=(\roman{*})]
    \item $z^2=1$;
    \item $z \overline{t_i} = \overline{t_i} z $, \textnormal{for $1 \leq i \leq n-1$};
    \item $\overline{t_i}^2=z$, \textnormal{for $1 \leq i \leq n-1$};
    \item $(\overline{t_i}. \overline{t_{i+1}})^3 =z $, \textnormal{for $1 \leq i \leq n-2$};
    \item $\overline{t_i}. \overline{t_j} = z \overline{t_j} .\overline{t_i}$, \textnormal{for $|i-j| \geq 2$ and $1 \leq i,j \leq n-1$}.
\end{enumerate}

Then $U$ is a generalized representation group of $S_n$ with base normal subgroup $K = \langle z \rangle $. Moreover, if $t_i$ denotes the transposition $(i \hspace{5pt} i+1)$ in $S_n$, then the map
\begin{align*}
  \lambda \colon U &\longrightarrow S_n \\
  z &\longmapsto 1 \\
 \overline{t_i} &\longmapsto  t_i
\end{align*}

\noindent is surjective and has kernel $K$. 
Additionally, if $n \neq 6,7$, then a generalized representation group of $A_n$ is given by $V=\lambda^{-1}(A_n)=\langle z, \overline{t_1}. \overline{t_2}, \overline{t_1} .\overline{t_3}, \dots, \overline{t_1} .\overline{t_{n-1}} \rangle \leq U$.

\end{proposition}

\begin{proof}
See Schur's original paper \cite{S11} or \cite[Chapter 2]{HH92} for a more modern exposition regarding generalized representation groups of $S_n$. 
The $A_n$ case is dealt with in \cite[\S3]{macedo}.\end{proof}

\section{Applications to $A_n$ and $S_n$ extensions}\label{sect_applications}

In this section we apply the results of the preceding sections to study the HNP and weak approximation for norm one tori of $A_n$ and $S_n$ extensions. Throughout the section, we fix the following notation: $L/K/k$ is a tower of number fields such that $L/k$ is Galois and $G=\Gal(L/k)$ is isomorphic to $A_n$ or $S_n$ with $n \geq 4$. We set $H=\Gal(L/K)$. 
For any subgroup $G'$ of $G$, we denote by $F_{G/G'}$ a flasque module in a flasque resolution of the Chevalley module $J_{G/G'}$. Let $X/k$ be a smooth compactification of the torus $T=R^1_{K/k}\bbG_m$.
We use the isomorphism~\eqref{cs_flasque} in
Theorem~\ref{thmcs}
to identify $\operatorname{H}^1(k,\Pic \overline{X})$ with
$\operatorname{H}^1(G,F_{G/H}).$

\subsection{Results for general $n$} First, we complete the proof of Theorem~\ref{main0}. For $G\cong A_n$ or $S_n$, we have $\operatorname{H}^3(G,\Z)\cong \Z/2$, unless $G\cong A_6$ or $A_7$ in which case $\operatorname{H}^3(G,\Z)\cong \Z/6$. Therefore, in our proof of Theorem~\ref{main0}, we can apply Corollary~\ref{an_sn_p_part_intro} to deal with the odd order torsion. It remains to analyze the $2$-primary parts of $\mathfrak{K}(K/k)$ and $\operatorname{H}^1(G,F_{G/H}).$ We start with the simpler case where $|H|$ is odd.

\begin{proposition}\label{h_odd_2}
If $|H|$ is odd, then

\begin{enumerate}[label=(\roman{*})]
 
    \item\label{cc2} $\operatorname{H}^1(G,F_{G/H})_{(2)} = \Z/2$, and
    
      \item\label{cc1} $\mathfrak{K}(K/k)_{(2)} =  \mathfrak{K}(L/k)_{(2)}$ and $\mathfrak{K}(K/k)_{(2)}$ has size at most 2.
\end{enumerate}  \end{proposition}

\begin{proof}

\begin{enumerate}[label=(\roman{*})]
 \item This follows from Corollary~\ref{gp_containment_intro}\ref{no p middle gp}.
    \item This is a consequence of Theorem~\ref{mid_gp_general_norm} and isomorphism \eqref{eq:tate_knot} of Theorem~\ref{Tate}. 
   \qedhere
    \end{enumerate}
\end{proof} 

\begin{proof}[Proof of Theorem~\ref{main0} for $|H|$ odd]
We analyze the $p$-primary parts of the groups in Theorem~\ref{main0} for each prime $p$. For $p$ odd, apply Corollary~\ref{an_sn_p_part_intro} and use the fact that $\mathfrak{K}(L/k)^\sim \hookrightarrow \operatorname{H}^3(G,\Z) = \Z/2$ (Theorem~\ref{Tate}). For $p=2$, use Proposition~\ref{h_odd_2}. By Theorem~\ref{thm:unramified_description_intro}, $\mathfrak{F}_{nr}(L/K/k) = \mathfrak{F}(G,H)$ is a subquotient of $H \cap [G,G]$, whereby $\mathfrak{F}_{nr}(L/K/k)_{(2)}=1$, since $|H|$ is odd.
As $\mathfrak{F}_{nr}(L/K/k)$ surjects onto $\mathfrak{F}(L/K/k)$, we also have $\mathfrak{F}(L/K/k)_{(2)}=1$.
\end{proof}

We now solve the case where $|H|$ is even. For this, we will use the generalized representation group $\overline{G}$ of $G$, the projection map $\lambda$ and the base normal subgroup $K=\langle z \rangle$ presented in Proposition~\ref{pres}, so our next two results do not apply when $G\cong A_6$ or $A_7$.

\begin{lemma}\label{aux_v4}
Suppose that $G$ is not isomorphic to $A_6$ or $A_7$ and that $|H|$ is even.  Let $h \in H$ be any element of order $2$. Then there exists a copy $A$ of $V_4$ inside $G$ such that 

\begin{itemize}
    \item $h \in A$;
    \item $z \in [\lambda^{-1}(A),\lambda^{-1}(A)]$.
\end{itemize} 
\end{lemma}

\begin{proof}

\textbf{Case 1) $h$ comprises a single transposition.} 
Relabeling if necessary, we can assume that $h=(1 \hspace{5pt} 2)$. Take $A=\langle (1 \hspace{5pt} 2),(3 \hspace{5pt} 4) \rangle$ and note that $[\lambda^{-1} \left( (1 \hspace{5pt} 2) \right),\lambda^{-1} \left((3 \hspace{5pt} 4) \right) ]=[\overline{t_1}, \overline{t_3}]$ in the notation of Proposition \ref{pres}. Using the relations satisfied by the elements $\overline{t_i} \in \overline{G}$ given in Proposition \ref{pres}, it is clear that this commutator is equal to $z$, as desired.

\textbf{Case 2) $h$ comprises more than one transposition.} Relabeling if necessary, we can assume that $h=(1 \hspace{5pt} 2)(3 \hspace{5pt} 4) \cdots (n-1 \hspace{5pt} n) $ for some even $n \geq 4$. Take $A=\langle h,x \rangle$, where $x=(1 \hspace{5pt} 3)(2 \hspace{5pt} 4)$ and let us prove by induction that $z=[\lambda^{-1}(h),\lambda^{-1}(x)]$. Note that, in the notation of Proposition \ref{pres}, we have $h= t_1 . t_3. \cdots . t_{n-1}$ and $x=t_2.t_1.t_2.t_3.t_2.t_3$.

\textbf{Base case $n=4$:} A straightforward (but long) computation using the relations satisfied by the elements $\overline{t_i}$ given in Proposition \ref{pres} shows that $[\lambda^{-1}(h),\lambda^{-1}(x)]=[\overline{t_1}.\overline{t_3},\overline{t_2}.\overline{t_1}.\overline{t_2}.\overline{t_3}.\overline{t_2}.\overline{t_3}]=z$.

\textbf{Inductive step:} Suppose that $h=(1 \hspace{5pt} 2)(3 \hspace{5pt} 4) \cdots (n-1 \hspace{5pt} n)(n+1 \hspace{5pt} n+2)$. Denoting the permutation $(1 \hspace{5pt} 2)(3 \hspace{5pt} 4) \cdots (n-1 \hspace{5pt} n)$ by $\tilde{h}$, write $h=\tilde{h}. t_{n+1}$. Now \[[\lambda^{-1}(h),\lambda^{-1}(x)]=[\lambda^{-1}(\tilde{h}) \overline{t_{n+1}},\lambda^{-1}(x)]=[\lambda^{-1}(\tilde{h}),\lambda^{-1}(x)]^{\overline{t_{n+1}}} [\overline{t_{n+1}},\lambda^{-1}(x)].\]
By the inductive hypothesis and the relations of Proposition \ref{pres}, $[\lambda^{-1}(\tilde{h}),\lambda^{-1}(x)]^{\overline{t_{n+1}}}=z^{\overline{t_{n+1}}}=z$ and $[\overline{t_{n+1}},\lambda^{-1}(x)]=[\overline{t_{n+1}},\overline{t_2}.\overline{t_1}.\overline{t_2}.\overline{t_3}.\overline{t_2}.\overline{t_3}]=1$, as desired.\end{proof}

The next proposition completes the proof of Theorem~\ref{main0}.

\begin{proposition}\label{1obs_even}
Suppose that $G$ is not isomorphic to $A_6$ or $A_7$ and that $|H|$ is even. Then
\begin{enumerate}[label=(\roman{*})]
    \item\label{even_cond2} $\operatorname{H}^1(G,F_{G/H})^\sim = \mathfrak{F}_{nr}(L/K/k)$;
    \item\label{even_cond1} $\mathfrak{K}(K/k)= \mathfrak{F}(L/K/k)$.
\end{enumerate}
\end{proposition}

\begin{proof}
\begin{enumerate}[label=(\roman{*}), leftmargin=*] 
\item By Theorems~\ref{thm:unramified_description_intro}, \ref{mid_gp_generalized} and isomorphism \eqref{cs_flasque} of Theorem~\ref{thmcs}, if we can show that $\mathfrak{F}(\overline{G},\overline{H}) \cong \mathfrak{F}({G},{H})$ then it will follow that the natural surjection $\operatorname{H}^1(G,F_{G/H})^\sim \twoheadrightarrow \mathfrak{F}_{nr}(L/K/k)$ is an isomorphism.
By Lemma~\ref{summary_objects}, it suffices to check that $\Ker \lambda \subset \Phi^{\overline{G}}(\overline{H})$, i.e.~that $z \in \Phi^{\overline{G}}(\overline{H})$. Let $A=\langle h,x\rangle$ be the copy of $V_4$ constructed in the proof of Lemma \ref{aux_v4}. Then $h \in H \cap x H x^{-1}$ and therefore $z = [\lambda^{-1}(h),\lambda^{-1}(x)] \in \Phi^{\overline{G}}(\overline{H})$, as desired.
\item By the isomorphism \eqref{cs_flasque} of Theorem~\ref{thmcs}, the statement in \ref{even_cond2} implies that the map $f_{L/K}$ in diagram~\eqref{eq:first_obs} is trivial. As this diagram is commutative, it follows that $g_{L/K}$ is also trivial and thus $\mathfrak{K}(K/k) = \Sha(T) = \Coker(g_{L/K})=\mathfrak{F}(L/K/k)$.\qedhere\end{enumerate}\end{proof}

Now that we have proved Theorem~\ref{main0}, we have reduced the study of the HNP and weak approximation for norm one tori of $A_n$ and $S_n$ extensions to a purely computational problem (except in the cases of $A_6$ and $A_7$). The groups $\mathfrak{F}(L/K/k)$ and $\mathfrak{K}(L/k)$ can be computed using the GAP algorithms described in Remark~\ref{1obs_alg} and
at the end of Section~\ref{sec:comp_met} below. The calculations of the knot group and of $\operatorname{H}^1(k,\Pic \overline{X})$ in the remaining cases where $G\cong A_6, A_7$ are done in Section \ref{sec:casea6}.

\begin{remark}
The method employed in this section to provide explicit and computable formulae for the knot group and the invariant $\operatorname{H}^1(k,\Pic \overline{X})$ in $A_n$ and $S_n$ extensions works for other families of extensions. For example, let $G'$ be any finite group such that $\operatorname{H}^3(G',\Z) = \Z/2$. Embed $G'$ into $S_n$ for some $n$ and suppose that $G'$ contains a copy of $V_4$ conjugate to $\langle (1,2)(3,4),(1,3)(2,4) \rangle$. For such a group $G'$, analogues of Lemma \ref{aux_v4} and Propositions \ref{h_odd_2} and \ref{1obs_even}
yield a systematic approach to the study of the HNP and weak approximation for $G'$-extensions.
\end{remark}

We proceed by investigating the possible isomorphism classes of the finite abelian group $\mathfrak{F}(G,H)$ (and thus, by Theorems~\ref{thm:unramified_description_intro}, \ref{main0} and isomorphism \eqref{cs_flasque}, of the invariant $\operatorname{H}^1(G,F_{G/H})$ as well).

\begin{proposition}\label{sn_unr}
The group $\mathfrak{F}(S_n,H)$
is an elementary abelian $2$-group. Moreover, every elementary abelian $2$-group occurs as $\mathfrak{F}(S_n,H)$ for some $n$ and some $H \leq S_n$.
\end{proposition}

\begin{proof}
It suffices to prove that for every element $h \in H \cap [S_n,S_n]$, we have $h^2 \in \Phi^{S_n}(H)$. This is clear from the definition of $\Phi^{S_n}(H)$ because $h$ is conjugate to its inverse in $S_n$. The statement on the occurrence of every elementary abelian $2$-group is shown in Proposition~\ref{2group_occurs} below. 
\end{proof}

\begin{proposition}\label{an_unr}
The group $\mathfrak{F}(A_n,H)$ is either isomorphic to $C_3$ or an elementary abelian $2$-group. Moreover, every such possibility is realised for some choice of $n$ and $H$.

\end{proposition}

\begin{proof}
First, we claim that any element of even order in $\mathfrak{F}(A_n,H)$ is $2$-torsion. Let $h\in H$ 
have even order. By \cite{groupprops},
$h$ is $A_n$-conjugate to $h^{-1}$. Therefore $h^2\in\Phi^{A_n}(H)$, which proves the claim.

Next, we claim that any element of odd order in $\mathfrak{F}(A_n,H)$ is $3$-torsion. Let $h\in H$ be such that its image in $\mathfrak{F}(A_n,H)$ has odd order. Replacing $h$ by a suitable power, we may assume that $h$ itself has odd order, whereby $h$ is $S_n$-conjugate to $h^2$. By the pigeonhole principle, at least two of the three $S_n$-conjugate elements $h, h^{-1},h^2$ are $A_n$-conjugate. Therefore, 
at least one of $h^{-2}, h, h^{3}$ is in $\Phi^{A_n}(H)$. Since $h$ has odd order, we conclude that in all cases $h^3\in \Phi^{A_n}(H)$, whence the claim.

Next, we show that $\mathfrak{F}(A_n,H)_{(3)}$ is cyclic. 
Suppose for contradiction that the images in $\mathfrak{F}(A_n,H)$ of $h_1,h_2 \in H$ generate a copy of $C_3\times C_3$. Replacing $h_1$ and $h_2$ by suitable powers if necessary, we may assume that the lengths of the cycles making up $h_1$ and $h_2$ are powers of $3$, say $3^{r_1} \leq 3^{r_2} \leq \cdots \leq 3^{r_k}$ for $h_1$ and $3^{s_1} \leq 3^{s_2} \leq \cdots \leq 3^{s_l}$ for $h_2$, where $k,l \geq 1$ and $r_i,s_j \in \Z_{\geq 0}$. Note that $h_1$ and $h_1^{-1}$ cannot be $A_n$-conjugate, or else we would have $h_1^2 \in \Phi^{A_n}(H)$, and similarly for $h_2$. The criterion \cite{groupprops} for an element of $A_n$ to be conjugate to its inverse yields $3^{r_i} \neq 3^{r_j}$ and $3^{s_i} \neq 3^{s_j}$ for $i \neq j$. Since $n=\sum\limits_{i=1}^{k} 3^{r_i}=\sum\limits_{i=1}^{l} 3^{s_i}$, the uniqueness of the representation of $n$ in base $3$ implies that $k=l$ and $r_i=s_i$ for every $i$. Thus the cycle structures of $h_1$ and $h_2$ are identical and hence $h_1 ,h_2$ and $h_2^2$ are conjugate in $S_n$. Therefore, at least two of these elements are $A_n$-conjugate, whereby at least one of $h_1^{-1}h_2,h_1^{-1}h_2^2, h_2$ is in $\Phi^{A_n}(H)$. This contradicts the assumption that the images of $h_1$ and $h_2$ generate a non-cyclic subgroup of $\mathfrak{F}(A_n,H)$. One can compute that $\mathfrak{F}(A_{12},H) \cong C_3$ for $H=\langle (1,2,3)(4,5,6,7,8,9,10,11,12) \rangle$ using GAP, for example. The statement on the occurrence of every elementary abelian $2$-group is shown in Proposition~\ref{2group_occurs} below.
 \end{proof}

\begin{proposition}\label{2group_occurs}
For every $k \geq 0$, there exists $n$ and a subgroup $H$ of $A_n$ such that \[\mathfrak{F}(A_n,H)_{(2)}\cong \mathfrak{F}(S_n,H)_{(2)} \cong C_2^k.\] 
\end{proposition}

\begin{proof}
The case $k=0$ is realised by letting $H=1$. From now on, assume that $k\geq 1$.
Let $H$ be generated by $k$ commuting and even permutations of order $2$ such that, for any $x,y \in H$ with $x \neq y$, the permutations $x$ and $y$ have distinct cycle structures. We define such a group recursively as $H=H_k$, starting from $H_1 = \langle (1,2)(3,4) \rangle$, $H_2=\langle (1,2)(3,4),(5,6)(7,8)(9,10)(11,12) \rangle$ and adding, at step $i$, a new generator $h_i$ such that:

\begin{itemize}
    \item $h_i$ is an even permutation of order $2$;
    \item $h_i$ is disjoint to the previous generators $h_1,\dots, h_{i-1}$;
    \item $h_i$ moves enough points so that its product with any element of $H_{i-1}$ has cycle structure different from that of any element of $H_{i-1}$.
\end{itemize}

\noindent Let $n$ be large enough so that $H\subset A_n$. It is straightforward to check that one then has $\Phi^{A_n}(H)=\Phi^{S_n}(H)=1$. Therefore, $\mathfrak{F}(A_n,H)=H \cap [A_n,A_n] = H \cong C_2^{k} $ and similarly for $\mathfrak{F}(S_n,H)$.\end{proof}

As a consequence of the work done so far, we can now establish Theorem \ref{thm:an_sn_options} and Corollary~\ref{cor:thorough}.

\begin{proof}[Proof of Theorem \ref{thm:an_sn_options}]
For $G\not\cong A_6$ or $A_7$ the results follow from Theorems~\ref{thm:unramified_description_intro} and \ref{main0} and Propositions~\ref{sn_unr} and \ref{an_unr}. For the $A_6$ and $A_7$ cases, we describe how to compute $\operatorname{H}^1(k,\Pic \overline{X})$ in Section~\ref{sec:casea6} -- the results of these computations are in Tables~\ref{a6_table} and \ref{a7_table} of the Appendix and the $C_3$ and $C_6$ cases occur therein.\end{proof}

\begin{proof}[Proof of Corollary \ref{cor:thorough}]
Theorem~\ref{thm:an_sn_options} shows that $\operatorname{H}^1(k,\Pic \overline{X})_{(p)}=0$ for a prime $p > 3$ and that $\operatorname{H}^1(k,\Pic \overline{X})_{(3)}=0$ if $G\cong S_n$. 
Theorem~\ref{thm:unramified_description_intro} gives $\mathfrak{F}_{nr}(L/K/k)=\mathfrak{F}(G,H)$. By Theorem~\ref{thm:an_sn_options}, $\operatorname{H}^1(k,\Pic \overline{X})^{\sim}_{(3)}$ is $3$-torsion, so Theorem~\ref{main0} gives $\operatorname{H}^1(k,\Pic \overline{X})^{\sim}_{(3)}=\mathfrak{F}(G,H)[3]$. 
Let $K_3=L^{H_3}$ and let $X_3$ be a smooth compactification of $R^1_{K_3/k} \bbG_m$. Now Corollary~\ref{cor:mid_gp_general_intro} and Theorem~\ref{main0} give $\operatorname{H}^1(k,\Pic \overline{X})^{\sim}_{(3)}\cong\operatorname{H}^1(k,\Pic \overline{X_3})^{\sim}_{(3)}=\mathfrak{F}(G,H_3)$.
If $|H|$ is odd then $\mathfrak{F}(G,H)_{(2)}$ is trivial and hence $\operatorname{H}^1(k,\Pic \overline{X})^{\sim}_{(2)}=\Z/2$ by Theorem~\ref{main0}. The result for $\operatorname{H}^1(k,\Pic \overline{X})^{\sim}_{(2)}$ when $|H|$ is even is obtained in a similar way to the result for the $3$-primary part.
\end{proof}

The following corollary of Theorem~\ref{thm:an_sn_options} and Corollary~\ref{cor:mid_gp_general_intro} gives a useful shortcut when analyzing the HNP and weak approximation for $S_n$ extensions, enabling one to reduce to the case where $H$ is a $2$-group. 

\begin{corollary}\label{cor:S_nSylow2}
Suppose that $G\cong S_n$, let $H_2$ be a Sylow $2$-subgroup of $H$ and let $K_2$ denote its fixed field. Let $X_2$ be a smooth compactification of $T_2=R^{1}_{K_2/k} \mathbb{G}_m$. Then we obtain a commutative diagram with exact rows as follows, where the vertical isomorphisms are induced by the natural inclusion $T\hookrightarrow T_2$:
\[
\xymatrix{0 \ar[r]& A(T) \ar[r]\ar[d]_{\cong}& \operatorname{H}^1(k,\operatorname{Pic}\overline{X})^{\sim} \ar[r]\ar[d]_{\cong} &\Sha(T)\ar[d]_{\cong} \ar[r]& 0\\
0 \ar[r]& A(T_2)\ar[r]& \operatorname{H}^1(k,\operatorname{Pic}\overline{X_2})^{\sim} \ar[r] &\Sha(T_2) \ar[r]& 0.
}
\]
Alternatively, the norm map $N_{K_2/K}:T_2 \twoheadrightarrow T$ can be used to obtain a similar commutative diagram with the direction of the vertical isomorphisms reversed.
\end{corollary}

\begin{remark}
Corollary~\ref{cor:S_nSylow2} also holds in the case $G\cong A_n$ provided $n\neq 6,7$ and $\mathfrak{F}(G,H)_{(3)}=1$. In Proposition~\ref{characterization_3torsion} we show that for most $n$ we have $\mathfrak{F}(A_n,H)_{(3)}=1$ for all subgroups $H$.
\end{remark}

The next lemma will aid our characterization of the existence of elements of order $3$ in $\mathfrak{F}(A_n,H)$. 

\begin{lemma}\label{3_part_aux_lemma}
Let $n=3^{l}$ for some $l \geq 0$ and let $\rho=(a_1 \cdots \hspace{5pt} a_{3^l})$ be a $3^{l}$-cycle in $S_n$. Let $j\in\bbZ$ with $j\equiv -1 \pmod{3}$. Then $\rho^j$ is $A_n$-conjugate to $\rho$ if and only if $l$ is even. 
\end{lemma} 

\begin{proof}
Observe that $\rho^{j}(a_i)=a_{i+j}$, where the subscripts are considered modulo $3^l$. Therefore, the permutation $x \in S_n$ defined by $x(a_i)=a_{1+(i-1)j }$ satisfies $x\rho x^{-1}=\rho^j$. Let $C$ be the $A_n$-conjugacy class of $\rho$. Since the $S_n$-conjugacy class of $\rho$ splits as a disjoint union $C \sqcup g C g^{-1}$ for any $g \in S_n \setminus A_n$, it is enough to show that $x\in A_n$ if and only if $l$ is even.
We study the cycle structure of $x$ by analyzing the fixed points of its powers. Observe that $x^{t}(a_i)=a_{1+(i-1)j^{t}}$ for every $t \geq 0$ and so \[x^t(a_i)=a_i \Leftrightarrow 1+(i-1)j^t \equiv i  \pmod{3^{l}} \Leftrightarrow (i-1)(j^t-1) \equiv 0 \pmod{3^{l}}.\] 
Therefore, the number of fixed points of $x^t$ is $\gcd(3^{l},j^{t}-1)$. Using this fact, we note two useful properties of the cycles occurring in a disjoint cycle decomposition of $x$:

\begin{enumerate}[label=(\roman{*}), leftmargin=*]
    \item\label{3_part_aux_lemma_step1} \textbf{The only cycle of }\bm{$x$}\textbf{ with odd length corresponds to the fixed point} \bm{$a_1$}\textbf{:} It suffices to show that, for odd $t \geq 1$, the only fixed point of $x^t$ is $a_1$. As $j \equiv -1 \pmod{3}$, it is easy to see that $j^{t}-1 \not\equiv 0 \pmod{3}$ for odd $t$ and thus $\gcd(3^{l},j^{t}-1)=1$.
    
    \item\label{3_part_aux_lemma_step2} \bm{$x$}\textbf{ does not contain a cycle with length divisible by $4$}\textbf{:} It is enough to prove that, for any $m \geq 1$, the number of fixed points of $x^{4m}$ and $x^{2m}$ coincide, i.e. that $\gcd(3^{l},j^{4m}-1)=\gcd(3^{l},j^{2m}-1)$. This is clear since $j^{4m}-1=(j^{2m}-1)(j^{2m}+1)$ and $j^{2m}+1 \not\equiv 0 \pmod{3}$.
    
\end{enumerate}
Let $c_1\cdot\ldots \cdot c_k$ be a disjoint cycle decomposition of $x$ where the cycle $c_i$ has length $|c_i|$. By \ref{3_part_aux_lemma_step1} and \ref{3_part_aux_lemma_step2}, we may assume that $|c_1|=1$ and $|{c_i}|\equiv 2 \pmod{4}$ for all $i\geq 2$. Note that $x\in A_n$ if and only if $k$ is odd.
Now $3^{l}= \sum_i |{c_i}|\equiv 1+\sum\limits_{i\geq 2} 2 \pmod{4}.$
Thus, $x\in A_n$  if and only if $3^l\equiv 1\pmod{4}$. 
\end{proof}

\begin{proposition}\label{characterization_3torsion}
There exists $H\leq A_n$ such that $\mathfrak{F}(A_n,H)_{(3)}\cong C_3$ if and only if $n \geq 5$ and $n = \sum\limits_{i=1}^{k} 3^{r_i}$ with $0\leq r_1< \dots <r_k$ and $|\{i \mid  r_i \text{ is odd}\}|$ is odd. 
\end{proposition}

\begin{proof}
Suppose that $\mathfrak{F}(A_n,H)_{(3)}\cong C_3$. It is easy to check that $\mathfrak{F}(A_4,H)_{(3)}=1$ for all $H \leq A_4$ so $n \geq 5$. Let $h$ be an element of $H$ such that its image in $\mathfrak{F}(A_n,H)$ generates $\mathfrak{F}(A_n,H)_{(3)}$. Replacing $h$ by a suitable power if necessary, we may assume that the lengths of the cycles making up $h$ are powers of $3$, say $3^{r_1} \leq 3^{r_2} \leq \cdots \leq 3^{r_k}$ with $r_i \in \Z_{\geq 0}$. 
If $h$ were $A_n$-conjugate to $h^{-1}$ then we would obtain $h\in \Phi^{A_n}(H)$, a contradiction. Therefore, by criterion \cite{groupprops} we have $3^{r_i} \neq 3^{r_j}$ for $i \neq j$ and $\sum\limits_{i=1}^{k} \frac{3^{r_i}-1}{2}$ is odd, i.e.~the number of odd $r_i$ is odd.

Conversely, assume that $n \geq 5$ is equal to $\sum\limits_{i=1}^{k} 3^{r_i}$ with $r_1< r_2 < \dots <r_k$ and $|\{i \mid  r_i \text{ is odd}\}|$ odd
and let $H$ be the cyclic group of order $3^{r_k}$ generated by $h$, where $$h=\underbrace{(1 \hspace{5pt} \cdots \hspace{5pt} 3^{r_1})}_{c_1} \underbrace{(3^{r_1}+1 \hspace{5pt} \cdots\hspace{5pt} 3^{r_1}+ 3^{r_{2}})}_{c_2} \dots  \underbrace{(\sum\limits_{i=1}^{k-1}3^{r_i}+1  \hspace{5pt} \cdots\hspace{5pt} n)}_{c_k}.$$

\noindent We will prove that $\mathfrak{F}(A_n,H)_{(3)} \cong C_3$. By Proposition~\ref{an_unr}, it is enough to show that $h\notin \Phi^{A_n}(H)$. Observe that $\Phi^{A_n}(H)$ is generated by elements of the form $h^{s-t}$ where $h^s$ is $A_n$-conjugate to $h^t$. We complete the proof by showing that $\Phi^{A_n}(H)\subset \langle h^3\rangle$. Suppose that $h^s$ is $A_n$-conjugate to $h^t$. We claim that $s\equiv t\pmod{3}$. Since conjugate elements have the same order, $3\mid s$ if and only if $3\mid t$. Now assume that $3\nmid s$. Then $h^s$ generates $H$ and has the same cycle type as $h$ so, relabelling if necessary, we may assume that $s=1$.
Suppose for contradiction that $t\equiv -1\pmod{3}$. For every $1 \leq i \leq k$, let $x_i \in S_n$ be such that $x_i$ only moves points appearing in $c_i$ and $x_ic_ix_i^{-1}=c_i^t$. Then $x=x_1 \cdot \ldots  \cdot x_k$ satisfies $xhx^{-1}=h^t$. Lemma~\ref{3_part_aux_lemma} shows that $x_i\in A_n$ if and only if $r_i$ is even. Since $|\{i \mid  r_i \text{ is odd}\}|$ is odd, $x\in S_n\setminus A_n$. This gives the desired contradiction as the $S_n$-conjugacy class of $h$ splits as a disjoint union $C\sqcup xCx^{-1}$ where $C$ denotes the $A_n$-conjugacy class of $h$.\end{proof}

\begin{remark}
For fixed $n$, it would be interesting to determine the list of isomorphism classes of $\mathfrak{F}(A_n,H)_{(2)}$ or $\mathfrak{F}(S_n,H)_{(2)}$ as $H$ ranges through the subgroups of $A_n$ or $S_n$, respectively. We give some observations regarding this problem without proof:
\begin{itemize}
    \item One can restrict the focus to $A_n$ since $\mathfrak{F}(A_n,H)_{(2)} \cong \mathfrak{F}(S_n,H)_{(2)}$.
    \item One can assume that $H$ is a $2$-group as $\mathfrak{F}(A_n,H)_{(2)} \cong \mathfrak{F}(A_n,H_2)$. 
    \item If $\mathfrak{F}(A_n,H)_{(2)} \cong C_2^k$ for some $k \in \Z_{\geq 0}$, then $k \leq d(H)$, where $d(H)$ denotes the minimal number of generators of $H$; in particular, it follows that $k \leq \frac{n}{2}$.
    \item If $\tilde{H}$ is a subgroup of $H$ of index $2$, then $\frac{1}{2}|\mathfrak{F}(A_n,\tilde{H})_{(2)}|\leq  |\mathfrak{F}(A_n,H)_{(2)}| \leq 2|\mathfrak{F}(A_n,\tilde{H})_{(2)}|$.
    \item If $\mathfrak{F}(A_{n_0},H)_{(2)} \cong C_2^k$ for some $n_0 \geq 1$ and $k \in \Z_{\geq 0}$, then $\mathfrak{F}(A_n,H)_{(2)} \cong C_2^k$ for all $n \geq n_0$.
    \item One has $\mathfrak{F}(A_n,H)_{(2)} \in \{1,C_2\}$ for all $n \leq 11$ and $H \leq A_n$ and $\mathfrak{F}(A_n,H)_{(2)} \in \{1,C_2,C_2^2\}$ for $n=12,13,14$ and all $H \leq A_n$.
\end{itemize}
\end{remark}

\subsection{Computational methods and results for small $n$}\label{sec:comp_met}
In this section we prove Propositions~\ref{main}, \ref{weakap}, Corollary~\ref{cor:A4} and Theorem~\ref{thm:malle_thm}. 
In order to prove Proposition~\ref{weakap}, 
we must compute the groups $\operatorname{H}^1(k,\Pic \overline{X})$ where $X$ is a smooth compactification of the norm one torus $R^1_{K/k}\bbG_m$ and $K/k$ is contained in a Galois extension $L/k$ with $\Gal(L/k)=G\cong  S_4, S_5,A_4,A_5$. 
One method to achieve this is via the identification $\operatorname{H}^1(k,\Pic \overline{X}) = \operatorname{H}^1(G,F_{G/H})$ in \eqref{cs_flasque}, where $H=\Gal(L/K)$. In \cite[\S5]{hoshi}, Hoshi and Yamasaki developed several algorithms in the computer algebra system GAP \cite{gap} to construct flasque resolutions. Using this work, one can compute the invariant $\operatorname{H}^1(G,F_{G/H})$ for low-degree field extensions, see e.g.~\cite[\S4]{macedo} for some examples. This computational method can also be used to prove Theorem \ref{thm:malle_thm}:

\begin{proof}[Proof of Theorem~\ref{thm:malle_thm}]

Note that an extension $K/k$ of degree $n$ is a $(G,H)$-extension (as defined on p.~\pageref{page_gh_extension}), where $G$ is a transitive subgroup of $S_n$ and $H$ is an index $n$ subgroup of $G$. 
Since there are a finite number of possibilities for $G$ and $H$, one can compute all possibilities for $\operatorname{H}^1(G,F_{G/H})$ using the aforementioned algorithms. If $\operatorname{H}^1(G,F_{G/H})=0$, then both the HNP for $K/k$ and weak approximation for $R^1_{K/k} \mathbb{G}_m$ hold by Theorem~\ref{thmvosk} and the isomorphism~\eqref{cs_flasque}. If $\operatorname{H}^1(G,F_{G/H})\neq 0$, one can compute the integer $\alpha(G)$ of Malle's conjecture and for every such case one obtains $\alpha(G) > 1$. Thus, if the conjecture holds, then the number of degree $n$ extensions with discriminant bounded by $X$ and for which the HNP or weak approximation fails is $o(X)$. The result then follows by observing that Malle's conjecture also implies that the number of degree $n$ extensions of $k$ with discriminant bounded by $X$ is asymptotically at least $c(k,n)X$ for some positive constant $c(k,n)$.\end{proof}

\begin{remark}\label{malle_proof_rmk}
We list a few observations about Theorem~\ref{thm:malle_thm} and its proof.
\begin{enumerate}[leftmargin=*, label=(\roman{*})]
    \item The reason for excluding degrees $n=8$ and $12$ is that in these cases there are pairs $(G,H)$, where $G \leq S_n$ is a transitive subgroup and $H$ is an index $n$ subgroup of $ G$, such that $\operatorname{H}^1(G,F_{G/H})$ is non-trivial and $\alpha(G)=1$. A more detailed analysis of the proportion of these $(G,H)$-extensions for which the local-global principles fail is needed in these cases. 
    \item Computing the values of $\alpha(G)$ for all transitive subgroups $G$ of $S_n$ with $\operatorname{H}^1(G, F_{G/H})\neq 0$ and
    $[G:H]=n$ 
    yields an upper bound (conditional on Malle's conjecture) on the number of degree $n$ extensions for which the HNP (or weak approximation for the norm one torus) fails. For example, the number of degree $14$ extensions of $k$ for which the HNP (or weak approximation for the norm one torus) fails is $\ll_{k,\epsilon}x^{\frac{1}{6}+\epsilon}$,
    when ordered by discriminant. 
    
  \item\label{malle_proof_rmk_item3} In the statement of Theorem \ref{thm:malle_thm} it suffices to assume Malle's conjecture only for the few transitive subgroups $G \leq S_n$ containing an index $n$ subgroup $H$ such that $\operatorname{H}^1(G,F_{G/H})$ is not trivial. Indeed, the assumption for all $G \leq S_n$ was used solely to show that the number of degree $n$ extensions of $k$ with discriminant bounded by $X$ is $\gg_{k,n} X$. For $n \leq 15$ composite, one can use an argument similar to that of \cite[pp.~723--724]{ellenberg} for $n$ even and the results of Datskovsky and Wright \cite{datskovsky} for cubics and of Bhargava, Shankar and Wang \cite{Bhargava2} for quintics to prove the aforementioned result. Finally, for $n$ prime we do not need any assumptions as the HNP for $K/k$ and weak approximation for $R^1_{K/k} \mathbb{G}_m$ always hold for extensions of prime degree (see \cite[Proposition 9.1 and Remark 9.3]{coll2}).
\item To simplify the statement we only presented results for degree $n\leq 15$ but one can obtain results for higher degrees in a similar way. However, Hoshi and Yamasaki's algorithms require one to embed the Galois group $G$ as a transitive subgroup of $S_n$, whereupon one quickly reaches the limit of the databases of such groups stored in computational algebra systems such as GAP. To overcome this problem, one can employ a modification of Hoshi and Yamasaki's algorithms written by the first author and made available at \cite{macedo_code}.
\end{enumerate}
\end{remark}

For most of our computational results, we did not employ the algorithms of Hoshi and Yamasaki and
instead used the formula of Theorem \ref{mid_gp_generalized} which expresses $\operatorname{H}^1(k,\Pic \overline{X})$ in terms of generalized representation groups of $G$ and $H$. We implemented this formula, along with the simplification afforded by Corollary~\ref{cor:mid_gp_general_intro}, 
as an algorithm in GAP (see \cite{macedo_code}). For the groups $G$ of Proposition~\ref{main}, our calculations were further simplified thanks to Theorem~\ref{main0}. The outcome of our computations appears in Tables \ref{a4_table} -- \ref{a7_table} of the Appendix. Proposition~\ref{weakap} follows immediately.

It is noteworthy to compare the two computational methods described above.
The approach based on Theorem~\ref{mid_gp_generalized} involves the computation of the focal subgroup $\Phi^{G}(H)$, which is generally fast for small subgroups $H$ but impractical for large ones. On the contrary, Hoshi and Yamasaki's method using flasque resolutions deals only with the $G$-module $J_{G/H}$, whose $\Z$-rank $\frac{|G|}{|H|}-1$ decreases as $|H|$ grows. 
Therefore this technique (or the modified version available at \cite{macedo_code}) is usually preferable when $H$ is large. In general, a combination of the two algorithms is the most convenient way to compute $\operatorname{H}^1(k,\Pic \overline{X})$ for all subgroups of a fixed group $G$.

We now move on to the proof of Proposition~\ref{main}. We use Theorem~\ref{main0} to reduce our task to the calculation of the first obstruction $\mathfrak{F}(L/K/k)$ and the knot group $\mathfrak{K}(L/k)$ for the Galois extension $L/k$. The former is achieved using the algorithm described in Remark~\ref{1obs_alg}. The computation of $\mathfrak{K}(L/k)$ follows from a simple application of isomorphism \eqref{eq:tate_knot} of Theorem~\ref{Tate} together with Proposition~\ref{prop:noptorsion} and Lemma~\ref{Res isom} below. Note that if $G=A_4,S_4,A_5$ or $S_5$ then $\operatorname{H}^3(G,\bbZ)\cong\bbZ/2$.

\begin{lemma}\label{Res isom}Let $G = A_4, S_4, A_5, S_5, A_6$ or $A_7$ and let $A$ be a copy of $V_4$ inside $G$. 
Then \[\Res^G_A:\operatorname{H}^{3}(G,\bbZ)_{(2)}\to \operatorname{H}^{3}(A,\bbZ)\] is an isomorphism.
\end{lemma}

\begin{proof}
Follows from the injectivity of $\Res^{D_4}_{V_4}:\operatorname{H}^{3}(D_4,\bbZ) \to \operatorname{H}^{3}(V_4,\bbZ)$ and Proposition \ref{prop:noptorsion}.
\end{proof}

More generally, the knot group of any Galois extension $L/k$ can be computed by combining the isomorphism \eqref{eq:tate_knot} of Theorem~\ref{Tate} and Lemma \ref{kernel_res}. We used these two results to implement an algorithm (available at \cite{macedo_code}) in GAP that, given the group $\Gal(L/k)$ and the list $l$ of decomposition groups $D_v$ at the ramified places, returns the knot group $\mathfrak{K}(L/k)$. 
We end this subsection by proving Corollary~\ref{cor:A4}.
\begin{proof}[Proof of Corollary~\ref{cor:A4}]
The isomorphisms $\mathfrak{K}(K/k)\cong \mathfrak{K}(L/k)$ and $\mathfrak{K}(L/k)\cong \mathfrak{K}(L/F)$ follow from Proposition~\ref{main} and isomorphism \eqref{eq:tate_knot} of Theorem~\ref{Tate}. The statement about weak approximation follows from the isomorphism  $\mathfrak{K}(K/k)\cong \mathfrak{K}(L/F)$, Voskresenski\u{\i}'s exact sequence \eqref{eq:Voskexact} and the fact that the middle group in this sequence is $\Z/2$ for both $R^1_{K/k} \bbG_m$ and $R^1_{L/F} \bbG_m$.
\end{proof}

\subsection{The $A_6$ and $A_7$ cases}\label{sec:casea6}

In this section we 
give a complete characterization of the Hasse norm principle and weak approximation for the norm one tori associated to $A_6$ and $A_7$ extensions.
Various subgroups of $A_6$ and $A_7$ are given by semidirect products of smaller subgroups. For brevity, we omit the precise construction of these semidirect products from the main text and refer the reader to Tables \ref{a6_table} and \ref{a7_table} of the Appendix containing the generators of these subgroups. Our main result is the following:

\begin{proposition}\label{thm:A6A7}

Suppose that $G$ is isomorphic to $A_6$ or $A_7$. 
Then $\mathfrak{K}(K/k)\hookrightarrow C_6$ and

\begin{itemize}
    \item $\mathfrak{K}(K/k)_{(2)}=1\iff \begin{cases} V_4\hookrightarrow H; \text{ or} \\ 
    C_4 \hookrightarrow H \text{ and } \exists\ v \text{ such that } D_4\hookrightarrow D_v; \text{ or}\\
    4  \nmid |H| \text{ and } \exists\ v \text{ such that } V_4\hookrightarrow D_v.
    \end{cases}$
    
    \item $\mathfrak{K}(K/k)_{(3)}=1\iff \begin{cases}
    C_3\hookrightarrow H; \text{ or}\\
    \exists\ v \text{ such that } C_3\times C_3\hookrightarrow D_v.
    \end{cases}$
\end{itemize}

\end{proposition}

We start by settling the Galois case of this proposition.

\begin{proposition}\label{a6_gal}

If $L/k$ is Galois with Galois group $A_6$ or $A_7$, then $\mathfrak{K}(L/k) \hookrightarrow C_6$ and

\begin{itemize}
    \item $\mathfrak{K}(L/k)_{(2)}=1$ if and only if there exists a place $v$ of $k$ such that $V_4 \hookrightarrow D_v$;
    \item $\mathfrak{K}(L/k)_{(3)}=1$ if and only if there exists a place $v$ of $k$ such that $C_3 \times C_3 \hookrightarrow D_v$.
    
\end{itemize}

\end{proposition}

\begin{proof}
This follows from isomorphism \eqref{eq:tate_knot} of Theorem~\ref{Tate}, Proposition~\ref{prop:noptorsion} and Lemma~\ref{Res isom}.
\end{proof}

We now solve the non-Galois case. As detailed in Section \ref{sec:comp_met}, we can compute the invariant $\operatorname{H}^1(k,\operatorname{Pic}\overline{X}) = \operatorname{H}^1(G,F_{G/H})$ for every possibility of $H=\Gal(L/K)$. The result of this computation is given in Tables \ref{a6_table} and \ref{a7_table} of the Appendix and it proves the following:

\begin{proposition}\label{thm:WAA6}
Suppose that $G$ is isomorphic to $A_6$ or $A_7$. Then
$\operatorname{H}^1(k,\Pic \overline{X})\hookrightarrow\Z/6$ and
\begin{itemize}
    \item $\operatorname{H}^1(k,\Pic \overline{X})_{(2)}=0 $ if and only if $V_4 \hookrightarrow H$;
    \item $\operatorname{H}^1(k,\Pic \overline{X})_{(3)}=0 $ if and only if $C_3 \hookrightarrow H$.
\end{itemize}
\end{proposition}

Building upon this proposition, we establish several results concerning the knot group $\mathfrak{K}(K/k)$. In particular, we immediately see that the invariant $\operatorname{H}^1(G,F_{G/H})$ is trivial if $H$ is isomorphic to $A_4$, $C_2 \times C_6$, $D_6$, $(C_6 \times C_2) \rtimes C_2$, $S_4$, $A_4 \times C_3$, $A_5$, $(A_4 \times C_3) \rtimes C_2$, $S_5$, $\operatorname{PSL}(3,2)$ or $A_6$. Thus, by Theorem~\ref{thmvosk} and isomorphism~\eqref{cs_flasque}, both groups $A(T)$ and $\mathfrak{K}(K/k)$ are trivial in all these cases.

Next, we investigate the cases where the first obstruction to the HNP for the tower $L/K/k$ coincides with the total obstruction, i.e.~the knot group. 
 
 \begin{proposition}\label{prop:6dividesH}
 If $6$ divides $|H|$, then $\mathfrak{K}(K/k)=\mathfrak{F}(L/K/k)$.
 \end{proposition}

 \begin{proof}
Let $G_1$ be a copy of $V_4$ inside $G$ such that $ H \cap G_1 \neq 1$ and $G_2$ a copy of $C_3 \times C_3$ inside $G$ such that $  H \cap G_2 \neq 1$. Set $H_i = H \cap G_i$ for $i=1,2$ and notice that the HNP holds for the extensions $L^{H_i}/L^{G_i}$ as they are of degree at most 3. Using Proposition~\ref{prop:noptorsion}, Lemma~\ref{Res isom} and duality, we find that the maps $\Cor^{G}_{G_1} : \hat{\operatorname{H}}^{-3}(G_1, \bbZ) \to \hat{\operatorname{H}}^{-3}(G, \bbZ)_{(2)}$ and $\Cor^{G}_{G_2} : \hat{\operatorname{H}}^{-3}(G_2, \bbZ) \to \hat{\operatorname{H}}^{-3}(G, \bbZ)_{(3)}$ are surjective. Hence
\[\Cor^{G}_{G_1} \oplus \Cor^{G}_{G_2} : \hat{\operatorname{H}}^{-3}(G_1, \bbZ) \oplus \hat{\operatorname{H}}^{-3}(G_2, \bbZ)  \to \hat{\operatorname{H}}^{-3}(G, \bbZ) \]
is surjective (recall that $\hat{\operatorname{H}}^{-3}(G, \bbZ) \cong \Z/6$) and therefore $\mathfrak{F}(L/K/k)=\mathfrak{K}(K/k)$ by Theorem~\ref{thm3DP}.
 \end{proof}

As a consequence of this result, one can use the GAP function \code{1obs} described in Remark \ref{1obs_alg} to computationally solve the cases where $6 \mid |H|$ and $\operatorname{H}^1(G,F_{G/H}) \neq 0$. The remaining possibilities for $H$ are dealt with in the two following results.

\begin{proposition}\label{a6_v4d4}
\begin{enumerate}[label=(\roman{*}), leftmargin=*]
    \item \label{V4}If $H \cong V_4$ or $ D_4$, then $\mathfrak{K}(K/k) \cong \mathfrak{K}(L/k)_{(3)}$;
    \item \label{C5}If $H \cong C_5$ or $C_7$, then $\mathfrak{K}(K/k) \cong \mathfrak{K}(L/k)$;
    \item \label{C3}If $H \cong C_3,C_3 \times C_3$ or $C_7 \rtimes C_3$, then $\mathfrak{K}(K/k) \cong \mathfrak{K}(L/k)_{(2)} $.
\end{enumerate}
    
\end{proposition}

\begin{proof}

We prove only \ref{V4} (\ref{C5} and \ref{C3} follow analogously). In this case $\operatorname{H}^1(G,F_{G/H}) = \Z/3$ (see Tables \ref{a6_table} and \ref{a7_table} of the Appendix) and thus $  \Z/3  \twoheadrightarrow \mathfrak{K}(K/k)$ by Theorem~\ref{thmvosk} and isomorphism~\eqref{cs_flasque}. The result now follows by Theorem~\ref{mid_gp_general_norm}, noting that ${d=[L:K]}=4$ or $8$ is coprime to $3$.
\end{proof}

 \begin{proposition}

\begin{enumerate}[label=(\roman{*}), leftmargin=*]
\item \label{C2}If $H \cong C_2$ or $D_5$, then $\mathfrak{K}(K/k) \cong \mathfrak{K}(L/k)$;
    \item \label{C4}If $H \cong C_4$ or $C_5 \rtimes C_4$, then \[\mathfrak{K}(K/k)  \cong \mathfrak{K}(L/k)_{(3)} \times \mathfrak{K}(M/k) \cong \mathfrak{K}(L/k)_{(3)} \times \mathfrak{F}(L/M/k),\] 
    \noindent where $M$ is the fixed field of a copy of $(C_3 \times C_3) \rtimes C_4$ inside $G$ containing $H_2 \cong C_4$. 
    \end{enumerate}
 \end{proposition}
 \begin{proof}First, note that in all cases $\mathfrak{K}(K/k)_{(3)}  \cong \mathfrak{K}(L/k)_{(3)}$, by Theorem~\ref{mid_gp_general_norm}. By Proposition~\ref{thm:WAA6} and Theorem~\ref{thmvosk}, it only remains to compute $\mathfrak{K}(K/k)_{(2)}$. For case~\ref{C2}, let $A$ be a copy of $S_3$ inside $G$ such that $A\cap H=H_2\cong C_2$ and let $F=L^A$ and $K_2=L^{H_2}$. Now Theorem~\ref{mid_gp_general_norm} shows that $\mathfrak{K}(K/k)_{(2)}\cong\mathfrak{K}(K_2/k)_{(2)}\cong \mathfrak{K}(F/k)_{(2)}$. Computing $\mathfrak{K}(F/k)_{(2)}$ using Proposition~\ref{prop:6dividesH} and the GAP function \code{1obs} described in Remark~\ref{1obs_alg} gives $\mathfrak{K}(F/k)_{(2)}\cong \mathfrak{K}(L/k)_{(2)}$, as required. For case~\ref{C4}, again let $K_2=L^{H_2}$. Then $\mathfrak{K}(K/k)_{(2)}\cong\mathfrak{K}(K_2/k)_{(2)}\cong \mathfrak{K}(M/k)_{(2)}$, by Theorem~\ref{mid_gp_general_norm}. Now Proposition~\ref{prop:6dividesH} gives $\mathfrak{K}(M/k)\cong \mathfrak{F}(L/M/k)$. Furthermore, Theorem~\ref{thmvosk} and isomorphism~\eqref{cs_flasque} combined with the results for $(C_3 \times C_3) \rtimes C_4$ in Tables~\ref{a6_table} and \ref{a7_table} of the Appendix show that $\mathfrak{K}(M/k)$ is $2$-torsion.\end{proof}

We have thus established the characterization of the HNP for an $A_6$ or $A_7$ extension given in Proposition \ref{thm:A6A7}. Using Proposition~\ref{thm:WAA6}, we can also give a full description of weak approximation. The local conditions controlling the validity of this principle are given in detail in the next theorem; they are a direct consequence of Propositions~\ref{thm:A6A7} and \ref{thm:WAA6} and  Voskresenski\u{\i}'s exact sequence \eqref{eq:Voskexact}.

 \begin{proposition}\label{thm:waconditionsA6}
Suppose that $G$ is isomorphic to $A_6$ or $A_7$.
\begin{itemize}

    \item If $V_4 \hookrightarrow H$ and $C_3 \hookrightarrow H$, then weak approximation holds for $T$.
    \item If $H \cong 1,C_2,C_5,C_7$ or $D_5$, then weak approximation holds for $T$ if and only if $V_4 \not\hookrightarrow D_v$ and $C_3 \times C_3 \not\hookrightarrow D_v$ for every place $v$ of $k$.
    \item If $H \cong C_4$ or $C_5 \rtimes C_4$, then weak approximation holds for $T$ if and only if $D_4 \not\hookrightarrow D_v$ and $C_3 \times C_3 \not\hookrightarrow D_v$ for every place $v$ of $k$.
    \item In all other cases, weak approximation holds for $T$ if and only if the HNP fails for $K/k$.
\end{itemize}
 \end{proposition}

\section{Examples}\label{sec:eg}

This section concerns the existence of number fields with prescribed Galois group for which the HNP holds, and the existence of those for which it fails. The main result is Theorem~\ref{examples_general_intro}. To prove it, we will use the notion of
$k$-adequate extensions, as introduced by Schacher in \cite{scha}.

\begin{definition}\label{def:admissible}
An extension $K/k$ of number fields is said to be \emph{$k$-adequate} if $K$ is a maximal subfield of a finite dimensional $k$-central division algebra. \end{definition}

A conjecture of Bartels (see \cite[p.~198]{Bartelsprime}) predicted that the HNP would hold for any $k$-adequate extension. This was proved by Gurak (see \cite[Theorem 3.1]{GurakcyclicSylow}) for Galois extensions, but disproved in general by Drakokhrust and Platonov (see \cite[\S9, \S11]{DP}). Given a Galois extension $L/k$, a result of Schacher (see \cite[Proposition 2.6]{scha}) shows that $L$ is $k$-adequate if and only if for every prime $p \mid [L:k]$ there are at least two places $v_1$ and $v_2$ of $k$ such that $D_{v_i}=\Gal(L_{v_i}/k_{v_i})$ contains a Sylow $p$-subgroup of $\Gal(L/k)$. This led Schacher to establish the following result:

\begin{theorem}\cite[Theorem 9.1]{scha}\label{thm:scha}
For any finite group $G$ there exists a number field $k$ and a \mbox{$k$-adequate} Galois extension 
$L/k$ with $\Gal(L/k) \cong G$.
\end{theorem}

We can now prove Theorem~\ref{examples_general_intro}, which
generalizes \cite[Corollary 3.3]{GurakcyclicSylow} to non-normal extensions.

\begin{proof}[Proof of Theorem~\ref{examples_general_intro}]
\begin{enumerate}[label=(\roman{*}), leftmargin=*]
    \item\label{exmp_suc} Let $L/k$ be a $k$-adequate Galois extension with Galois group $G$ as given in Theorem~\ref{thm:scha}. Let $K=L^H$ and $T=R^1_{K/k}\mathbb{G}_m$. Recall that, by Theorem~\ref{Tate}, 
    $$\Sha(T)^\sim = \Ker \left( \operatorname{H}^2(G,J_{G/H}) \xrightarrow{\Res}  \prod_{v \in \Omega_k}\mathrm{H}^2(D_v,J_{G/H}) \right).$$
    Let $p$ be a prime dividing $|G|$ and let $D_v$ be a decomposition group containing a Sylow $p$-subgroup of $G$.
   Then Proposition \ref{prop:noptorsion} and the transitivity of restriction show that the map $$\operatorname{H}^2(G,J_{G/H})_{(p)} \xrightarrow{\Res}  \prod_{v \in \Omega_k}\mathrm{H}^2(D_v,J_{G/H})$$ is injective. 
    It follows that $\Sha(T) = 0$ and so $\mathfrak{K}(K/k)$ is trivial. The statement regarding weak approximation follows from Theorem~\ref{thmvosk} and isomorphism~\eqref{cs_flasque} of Theorem~\ref{thmcs}.
    \item\label{exmp_fail} By \cite[Lemma 6]{DP}, there is a Galois extension $L/k$ of number fields with $\Gal(L/k) \cong G$ such that every decomposition group is cyclic. Let $K=L^H$, $T=R^1_{K/k}\mathbb{G}_m$ and let $X$ be a smooth compactification of $T$. By \cite[\S 3, Theorem 6 and Corollary 2]{Vosk}, we have $A(T)=0$ and $\Sha(T) \cong \operatorname{H}^1(k,\Pic \overline{X})^\sim$. The result now follows from isomorphism~\eqref{cs_flasque} of Theorem~\ref{thmcs} and the fact that $\mathfrak{K}(K/k) = \Sha(T)$.\qedhere
\end{enumerate}\end{proof}

As a consequence of the work done in the proof of Theorem~\ref{examples_general_intro}, we can also obtain a version of Theorem~\ref{thm:an_sn_options} for the knot group and the defect of weak approximation. In what follows, let $L/K/k$ be a tower of number fields where $L/k$ is Galois with Galois group $G \cong A_n$ or $S_n$ and let $T=R^1_{K/k}\bbG_m$.

\begin{proposition}\label{prop:an_sn_options_knotgp}

\begin{enumerate}[label=(\roman{*})]
\item For $G\cong S_n$ the groups $\mathfrak{K}(K/k)$ and $A(T)$ are elementary abelian $2$-groups.
Moreover, every possibility for $\mathfrak{K}(K/k)$ is realised: given an elementary abelian $2$-group $A$, there exists $n\in\bbN$ and an extension of number fields $K/k$ whose normal closure has Galois group $S_n$ such that $\mathfrak{K}(K/k)\cong A$.
Likewise, every possibility for $A(T)$ is realised.
\item For $G\cong A_n$ the groups $\mathfrak{K}(K/k)$ and $A(T)$ are elementary abelian $2$-groups or isomorphic to $C_3$ or $C_6$. Again, every possibility for $\mathfrak{K}(K/k)$ is realised, and likewise for $A(T)$.
\end{enumerate}
\end{proposition}

\begin{proof}
This follows from Theorems~\ref{thm:an_sn_options}, \ref{examples_general_intro} and \ref{thmvosk}.
\end{proof}

To conclude this section, we provide examples of number fields over $\Q$ illustrating that in every case addressed by Propositions~\ref{main} and \ref{thm:A6A7}, there exists an extension of the desired type satisfying the HNP. Furthermore, in the cases where failure of the HNP is theoretically possible, we construct examples showing that failures actually occur (over at most a quadratic extension of $\Q$). When looking for such examples, \cite[Lemmas 18 and 20]{Stern} give useful practical conditions to test the local properties of Proposition~\ref{main}. Some of these extensions were found using the LMFBD database \cite{LMFDB} and all assertions below concerning Galois groups and ramification properties were verified using the computer algebra system \textsc{magma} \cite{magma}.   

\subsection{Successes}

\begin{itemize}[leftmargin=*]
   \item First consider $G=A_4$ or $S_4$. Let $L/\Q$ be the splitting field of the polynomial $f(x)$ defined as
\[ f(x) = \begin{cases*}
                   x^4 - 2x^3 + 2x^2 + 2 & if $G=A_4$, \\
                   x^4 - 2x^3 - 4x^2 - 6x - 2  & if $G=S_4$.
                 \end{cases*} \]
\noindent In both cases $L/\Q$ is a Galois extension with Galois group $G$ such that the decomposition group at the prime $2$ is the full Galois group. Applying Proposition \ref{main} we thus conclude that the HNP holds for $L/\Q$ as well as for any subextension $K/\Q$ contained in $L/\Q$.
\item For $G=A_5$, let $K=\Q(\alpha)$, where $\alpha$ is a root of the polynomial $x^5 -x^4 +2x^2 -2x +2$, and let $L/\Q$ be the normal closure of $K/\Q$. We have $\Gal(L/\Q) \cong A_5$ and there exists a prime $\mathfrak{p}$ of $K$ above $2$ with ramification index $4$, so it follows that $4 \mid |D_2|$. 
Since any subgroup of $A_5$ with order divisible by $4$ contains a copy of $V_4$ generated by two double transpositions, Proposition \ref{main} shows that the HNP holds for any subextension of $L/\Q$. 

\item For $G=S_5$, take $K=\Q(\alpha)$, where $\alpha$ is a root of the polynomial $x^{10} - 4x^9 - 24x^8 + 80x^7 + 174x^6 - 416x^5 - 372x^4 + 400x^3 + 370x^2 + 32x - 16$, and let $L/\Q$ be the normal closure of $K/\Q$. One can verify that $\Gal(L/\Q) \cong S_5$ and that there is a prime $\mathfrak{p}$ of $K$ above $2$ with ramification index $8$. By the same reasoning as in the $A_5$ case, $D_2$ contains a copy of $V_4$ generated by two double transpositions, and thus the HNP holds for any subextension of $L/\Q$ by Proposition \ref{main}.
\item For $G=A_6$, let $K=\Q(\alpha)$, where $\alpha$ is a root of the polynomial $x^{15} - 3x^{13} - 2x^{12} + 12x^{10} + 50x^9 - 54x^7 + 68x^6 - 162x^5 + 30x^4 - 67x^3 + 15x + 4$, and let $L/\Q$ be the normal closure of $K/\Q$. We have $\Gal(L/\Q) \cong A_6$ and there are primes $\mathfrak{p}$ and $\mathfrak{q}$ of $K$ above $2$ and $3$, respectively, such that $[K_\mathfrak{p}:\mathbb{Q}_2]=8$ and $[K_\mathfrak{q}:\mathbb{Q}_3]=9$. Since every subgroup of $A_6$ with order divisible by $8$ contains a copy of $D_4$, it follows that $D_4 \hookrightarrow D_2$. Analogously, we have $C_3 \times C_3 \hookrightarrow D_3$. Proposition \ref{thm:A6A7} then shows that the HNP holds for any subextension of $L/\Q$.
\item For $G=A_7$, let $L/\Q$ be the splitting field of the polynomial $x^7-3x^6-3x^5-x^4+12x^3+24x^2+16x+24$. We have $\Gal(L/\Q) \cong A_7$ and the primes $2$ and $3$ ramify in $L/\Q$. Let $M$ be the fixed field of the subgroup $\langle  (2,3)(5,7), (1,2)(4,5,6,7), (2,3)(5,6) \rangle \cong (A_4 \times C_3) \rtimes C_2 $ of $A_7$, a degree $35$ extension of $\Q$. Given a prime $p$, let $e=e(p)$ denote its ramification index and $f=f(p)$ its inertial degree in $L$. 
Note that if the decomposition $\mathcal{O}_{M} / p\mathcal{O}_{M} \cong  \bigoplus\limits_{i}\mathbb{F}_{p^ {f_i}}[t_i] / (t_i^ {e_i})$ holds for some $e_i,f_i \in \Z_{\geq 0}$, then $\lcm(e_i) \mid e$, $\lcm(f_i) \mid f$ and hence $\lcm(e_i) \cdot \lcm(f_i) \mid ef = |D_p|$. Factoring the prime $p=2$ in $\mathcal{O}_M$ gives $\lcm(e_i)=12$ and $\lcm(f_i)=2$, so $24 \mid |D_2|$. Since any subgroup of $A_7$ with order divisible by $24$ contains a copy of $D_4$, we conclude that $D_4 \hookrightarrow D_2$. Using the same reasoning with the prime $p=3$, we find $18 \mid |D_3|$ and consequently $D_3$ contains a copy of $C_3 \times C_3$. By Proposition \ref{thm:A6A7}, it follows that the HNP holds for any subextension of $L/\Q$.
\end{itemize}

\begin{remark}

An alternative approach to find examples of number fields satisfying the HNP and with Galois groups as in Propositions~\ref{main} and \ref{thm:A6A7} is to use $\Q$-adequate extensions. 
Indeed, examining the local conditions of Propositions \ref{main} and \ref{thm:A6A7}, it is clear that the HNP holds for any subextension of a $\Q$-adequate Galois extension with Galois group $G=A_4,S_4,A_5,S_5,A_6,A_7$. The existence of $\Q$-adequate extensions with prescribed Galois group $G$ has been studied by Schacher and others. For $G=A_4,S_4,A_5,S_5,A_6,A_7$, there exist $\Q$-adequate Galois extensions $L/\Q$ with $\Gal(L/\Q) \cong G$. We give some references for the interested reader. For $G=A_4, A_5$ see \cite{exmp_a4}, \cite{exmp_a5}, respectively. In fact, for these two groups stronger results hold. For $G=A_4$ there exist $k$-adequate Galois extensions with Galois group $A_4$ for \emph{any} global field $k$ of characteristic not equal to $2$ or $3$ (see \cite[Corollary 2.2]{exmp_a4}). For $G=A_5$, \cite[Theorem 1]{exmp_a5} constructs $k$-adequate Galois extensions with Galois group $A_5$ for any number field $k$ such that $\sqrt{-1} \not\in k$. For $G=S_4, S_5$ see \cite[Theorem 7.1]{scha}. The cases $G=A_6,A_7$ are treated in \cite{exmp_a6a7}. We chose not to pursue this approach because the polynomials defining the field extensions were rather cumbersome, particularly for $A_6$ and $A_7$.
\end{remark}

\subsection{Failures}

\begin{itemize}[leftmargin=*]

\item We start with the cases where $G$ is $A_4$ or $S_4$. Let $L/\Q$ be the splitting field of $f(x)$, where
\[ f(x) = \begin{cases*}
                   x^4 + 3x^2 - 7x + 4 & if $G=A_4$, \\
                   x^4-x^3-4x^2+x+2  & if $G=S_4$.
                 \end{cases*} \]
\noindent In both cases $L/\Q$ is a Galois extension with Galois group $G$ such that every decomposition group is cyclic. Therefore, Proposition~\ref{main}
shows that the HNP fails for any subextension of $L / k$ falling under case (i) or (ii) of Proposition~\ref{main}, i.e.~an extension where the HNP can theoretically fail. 

\item We now find examples for the $A_5$ and $S_5$ cases using work of Uchida \cite{uchida}. Examples for the $A_6$ and $A_7$ cases can be obtained in a manner analogous to the construction for $A_5$. Let $F/\bbQ$ be the splitting field of $f(x) = x^5-x+1$ and set $D= \operatorname{Disc}(f)=19 \cdot 151$. By \cite[Corollary and Theorem 2]{uchida}, $F / \Q(\sqrt{D})$ is an unramified Galois extension with Galois group $A_5$, while $F(\sqrt{2})/\Q(\sqrt{2  D})$ is an unramified Galois extension with Galois group $S_5$. If $G=A_5$ then set $L=F,k=\Q(\sqrt{D})$. If $G=S_5$ then set $L=F(\sqrt{2}),k=\Q(\sqrt{2  D})$. Let $K/k$ be a subextension of $L / k$ falling under case (i) or (ii) of Proposition~\ref{main}. 
Since $L/k$ is unramified, all its decomposition groups are cyclic, whereby the HNP fails for $K/k$ by the criterion of Proposition~\ref{main}.

A similar construction allows us to provide examples of unramified Galois $A_6$ and $A_7$ extensions. By Proposition \ref{thm:A6A7}, these extensions have knot groups isomorphic to $C_6$ and therefore the HNP fails for them. It is also possible to construct failures with knot group $C_2$ or $C_3$. Indeed, if $G=A_6$ or $A_7$, one can set $S=C_3 \times C_3$ in \cite[Lemma 6]{DP} in order to get a Galois extension of number fields with decomposition group $D_v=C_3 \times C_3$ for every ramified place $v$. Since the remaining places have cyclic decomposition groups, it follows from Proposition~\ref{thm:A6A7} that the knot group of this extension is $C_2$. An analogous construction choosing $S=D_4$ gives a Galois extension of number fields with knot group equal to $C_3$.

\end{itemize}

\section*{Appendix}
We present the results of the computer calculations outlined in Section \ref{sec:comp_met}. In the following tables, we distinguish non-conjugate but isomorphic groups with a letter in front of the isomorphism class. 

\setcounter{table}{0}

\begin{table}[H]\caption{}
\label{a4_table}
\begin{tabular}{ |c|c|c| }
 \hline
 \multicolumn{3}{|c|}{\boldmath$G=A_4$} \\
 \hline
 $[K:k]$ & $H$  & $\operatorname{H}^1(G,F_{G/H})$\\
 \hline
 12 & $1$ & $\Z/2$   \\  \hline
 6 & $C_2=\langle (1,2)(3,4) \rangle$ & $\Z/2$ \\  \hline
 4 & $C_3=\langle (1,2,3) \rangle$ & $\Z/2$ \\ \hline
 3 & $V_4=\langle (1,2)(3,4),(1,3)(2,4) \rangle$ & $0$ \\
 \hline
\end{tabular}
\end{table}

\begin{table}[H]\caption{}
\label{s4_table}
\begin{tabular}{ |c|c|c|  }
 \hline
 \multicolumn{3}{|c|}{\boldmath$G=S_4$} \\
 \hline
 $[K:k]$ & $H$  & $\operatorname{H}^1(G,F_{G/H})$\\
 \hline
 24 & $1$ & $\Z/2$   \\  \hline
 12 & $C_2a=\langle (1 , 2) \rangle$ & $0$ \\  \hline
 12 & $C_2b=\langle (1 , 2)(3 , 4) \rangle$ & $\Z/2$ \\  \hline
 8 & $C_3= \langle (1,2,3) \rangle$ & $\Z/2$ \\  \hline
 6 & $C_4 = \langle (1,2,3,4) \rangle$ & $0$ \\   \hline
 6 & $V_4 = \langle (1, 2),(3 , 4) \rangle$ & $0$ \\   \hline
 6 & $V_4 = \langle (1, 2)(3 , 4),(1,3)(2,4) \rangle$ & $0$ \\   \hline
 4 & $S_3=\langle (1,2,3),(1,2) \rangle $ & $0$ \\  \hline 
 3 & $D_4=\langle (1,2,3,4),(1,3) \rangle $ & $0$ \\   \hline
 2 & $A_4=\langle (1,2)(3,4),(1,2,3) \rangle $ & $0$ \\   
 
 \hline
\end{tabular}
\end{table}

\begin{table}[H]\caption{}
\label{a5_table}
\begin{tabular}{ |c|c|c|  }
 \hline
 \multicolumn{3}{|c|}{\boldmath$G=A_5$} \\
 \hline
 $[K:k]$ & $H$  & $\operatorname{H}^1(G,F_{G/H})$\\
 \hline
 60 & $1$ & $\Z/2$   \\  \hline
 30 & $C_2=\langle (1,2)(3,4) \rangle$ & $\Z/2$ \\  \hline
 20 & $C_3=\langle (1,2,3) \rangle$ & $\Z/2$ \\  \hline
 15 & $V_4 = \langle (1,2)(3,4), (1,3)(2,4) \rangle$ & $0$ \\  \hline
 12 & $C_5=\langle (1,2,3,4,5) \rangle$ & $\Z/2$ \\   \hline
 10 & $S_3=\langle (1,2,3),(1,2)(4,5) \rangle$ & $\Z/2$ \\    \hline
 6 & $D_{5}=\langle (1,2,3,4,5),(2,5)(3,4) \rangle$ & $\Z/2$ \\ \hline
 5 & $A_4 = \langle (1,2)(3,4), (1,2,3)\rangle $ & $0$ \\ 
 \hline
\end{tabular}
\end{table}

\begin{table}[H]\caption{}
\label{s5_table}
\begin{tabular}{ |c|c|c|  }
 \hline
 \multicolumn{3}{|c|}{\boldmath$G=S_5$} \\
 \hline
 $[K:k]$ & $H$  & $\operatorname{H}^1(G,F_{G/H})$\\
 \hline
 120 & $1$ & $\Z/2$   \\  \hline
 60 & $C_2a=\langle (1 , 2) \rangle$ & $0$ \\  \hline
 60 & $C_2b=\langle (1 , 2)(3 , 4) \rangle$ &  $\Z/2$ \\  \hline
 40 & $C_3 = \langle (1,2,3) \rangle$ & $\Z/2$ \\  \hline
 30 & $C_4= \langle (1,2,3,4) \rangle$ & $0$ \\   \hline
  30 & $V_4a = \langle (1 , 2),(3 , 4) \rangle$ & $0$ \\    \hline
 30 & $V_4b = \langle (1 , 2)(3 , 4), (1, 3)(2 , 4) \rangle$ & $0$ \\ \hline
 24 & $C_5 = \langle (1,2,3,4,5) \rangle$ & $\Z/2$ \\   \hline
 20 & $C_6 = \langle (1,2,3),(4,5)\rangle $ & $0$ \\   \hline
 20 & $S_3a=\langle (1 , 2 , 3),(1 , 2) \rangle$ & $0$ \\   \hline
 20 & $S_3b=\langle (1 , 2 , 3),(1 , 2)(4 , 5) \rangle$ & $\Z/2$ \\   \hline
 15 & $D_4 = \langle (1,2,3,4),(1,3) \rangle$ & $0$ \\   \hline
 12 & $D_{5}=\langle (1,2,3,4,5) , (2,5)(3,4) \rangle$ & $\Z/2$ \\   \hline
 10 & $A_4 = \langle (1, 2)(3,4) , (1,2,3) \rangle$ & $0$ \\   \hline
 10 & $S_3\times C_2  = \langle (1,2,3),(1,2),(4,5) \rangle$ & $0$ \\   \hline
 6 & $C_5 \rtimes C_4=\langle (1,2,3,4,5),(2,3,5,4) \rangle$ & $0$ \\   \hline
 5 & $S_4=\langle (1,2,3,4),(1,2) \rangle$ & $0$ \\   
 \hline
  2 & $A_5=\langle (1,2,3,4,5), (1,2,3) \rangle$ & $0$ \\   \hline
\end{tabular}
\end{table}

\begin{table}[H]
\caption{}
\label{a6_table}
\begin{tabular}{ |c|c|c|  }
 \hline
 \multicolumn{3}{|c|}{\boldmath$G=A_6$} \\
 \hline
 $[K:k]$ & $H$  & $\operatorname{H}^1(G,F_{G/H})$\\
 \hline
 360 & $1$ & $\Z/6$   \\  \hline
 180 & $C_2=\langle (1,2)(3,4) \rangle$ & $\Z/6$ \\  \hline
 120 & $C_3=\langle (1,2,3)\rangle$ & $\Z/2$ \\  \hline
 120 & $C_3=\langle (1,2,3)(4,5,6) \rangle$ & $\Z/2$ \\  \hline
 90 & $C_4 = \langle  (1,2,3,4)(5,6) \rangle$ & $\Z/6$ \\  \hline
  90 & $V_4a=\langle (1,2)(3,4), (1,3)(2,4) \rangle$ &  $\Z/3$ \\  \hline
 90 & $V_4b=\langle (1,2)(5,6), (1,2)(3,4) \rangle$ &  $\Z/3$ \\  \hline
 72 & $C_5 = \langle (1,2,3,4,5) \rangle$ & $\Z/6$ \\  \hline
 60 & $S_3a=\langle (1,2,3)(4,5,6), (1,2)(4,5)\rangle$ & $\Z/2$ \\   \hline
 60 & $S_3b = \langle (1,2,3), (1,2)(4,5) \rangle$ & $\Z/2$ \\    \hline
45 & $D_4 =\langle (1,2,3,4)(5,6), (1,3)(5,6) \rangle$ & $\Z/3$ \\ \hline
 40 & $C_3 \times C_3 = \langle (1,2,3), (4,5,6) \rangle$ & $\Z/2$ \\   \hline
 36 & $D_{5}=\langle (1,2,3,4,5), (2,5)(3,4) \rangle$ & $\Z/6$ \\   \hline
 30 & $A_4a=	\langle (1,2)(3,4), (1,2,3) \rangle$ & $0$ \\   \hline
 30 & $A_4b =\langle (1,2,3)(4,5,6), (1,4)(2,5) \rangle$ & $0$ \\   \hline
 20 & $(C_3 \times C_3)\rtimes C_2 = \langle (1,2,3), (4,5,6), (1,2)(4,5) \rangle$ & $\Z/2$ \\   \hline
 15 & $S_4a=\langle (1,2,3,4)(5,6), (1,2)(5,6) \rangle$ & $0$ \\   \hline
 15 & $S_4b =\langle  (1,3,5)(2,4,6), (1,6)(2,5) \rangle$ & $0$ \\   \hline
 10 & $(C_3 \times C_3) \rtimes C_4=\langle  (1,2,3),(4,5,6), (1,4)(2,5,3,6) \rangle$ & $\Z/2$ \\   \hline
 6 & $A_5a=\langle (1,2,3,4,5), (1,2,3) \rangle$ & $0$ \\   \hline
 6 & $A_5b = 	\langle (1,2,3,4,5), (1,4)(5,6) \rangle$ & $0$ \\   
 \hline
\end{tabular}
\end{table}

\begin{table}[H]\caption{}
\label{a7_table}
\scalebox{0.95}{
\begin{tabular}{ |c|c|c|  }
 \hline
 \multicolumn{3}{|c|}{\boldmath$G=A_7$} \\
 \hline
 $[K:k]$ & $H$  & $\operatorname{H}^1(G,F_{G/H})$\\
 \hline
 2520 & $1$ & $\Z/6$   \\  \hline
 1260 & $C_2=\langle (1,2)(3,4) \rangle$ & $\Z/6$ \\  \hline
  840 & $C_3a=\langle (1,2,3)\rangle $ &  $\Z/2$ \\  \hline
 840 & $C_3b=\langle (1,2,3)(4,5,6) \rangle$ &  $\Z/2$ \\  \hline
  630 & $C_4 = \langle  (1,2,3,4)(5,6) \rangle$ & $\Z/6$ \\   \hline
  630 & $V_4a=\langle (1,2)(3,4), (1,3)(2,4) \rangle$ & $\Z/3$ \\  \hline
  630 & $V_4b=\langle (1,2)(5,6), (1,2)(3,4) \rangle$ & $\Z/3$ \\  \hline
 504 & $C_5= \langle (1,2,3,4,5) \rangle$ & $\Z/6$ \\    \hline
420 & $C_6=\langle  (1,2)(3,4)(5,6,7) \rangle$ & $\Z/2$ \\ \hline
 420 & $S_3a=\langle (1,2,3)(4,5,6), (1,2)(4,5)\rangle$ & $\Z/2$ \\   \hline
 420 & $S_3b= \langle (1,2,3), (1,2)(4,5) \rangle$ & $\Z/2$ \\   \hline
 360 & $C_7=\langle (1,2,3,4,5,6,7) \rangle$ & $\Z/6$ \\   \hline
 315 & $D_4=\langle (1,2,3,4)(5,6), (1,3)(5,6) \rangle$ & $\Z/3$ \\   \hline
 280 & $C_3 \times C_3 = \langle (1,2,3), (4,5,6) \rangle$ & $\Z/2$ \\   \hline
 252 & $D_{5}=\langle (1,2,3,4,5), (2,5)(3,4) \rangle$ & $\Z/6$ \\   \hline
 210 & $A_4a=	\langle (1,2)(3,4),(1,2,3)  \rangle$ & $0$ \\   \hline
 210 & $A_4b=\langle (1,2,3)(4,5,6),(1,4)(2,5) \rangle$ & $0$ \\   \hline
 210 & $A_4c=\langle (1,5,3)(4,7,6), (2,6)(4,7)\rangle$ & $0$ \\   \hline
 210 & $A_4d=\langle  (1,2,5)(4,6,7), (3,4)(6,7)  \rangle$ & $0$ \\  \hline
 210 & $C_2 \times C_6=\langle (1,2)(3,5)(4,6,7), (1,3)(2,5) \rangle$ & $0$ \\   \hline
  210 & $D_{6}=\langle  (1,2)(3,5)(4,6,7), (1,2)(6,7)  \rangle$ & $0$ \\  \hline
 210 & $C_3 \rtimes C_4=\langle (2,3,6), (1,4,7,5)(3,6)  \rangle$ & $\Z/2$   \\  \hline
 140 & $(C_3 \times C_3) \rtimes C_2= \langle (1,2,3), (4,5,6), (1,2)(4,5) \rangle$ & $\Z/2$ \\  \hline
  126 & $C_5 \rtimes C_4=\langle (1,2)(4,5,7,6), (3,6,7,4,5) \rangle$ &  $\Z/6$   \\  \hline
 120 & $C_7 \rtimes C_3=\langle (1,7,4,2,6,5,3), (2,3,5)(4,6,7)  \rangle$ &  $\Z/2$ \\  \hline
 105 & $(C_6 \times C_2) \rtimes C_2=\langle  (1,2)(3,5)(4,6,7), (1,3)(2,5), (1,2)(6,7) \rangle$ & $0$ \\  \hline
  105 & $S_4a=\langle  (1,2,3,4)(5,6),(1,2)(5,6) \rangle$ & $0$ \\  \hline
 105 & $S_4b=\langle  (1,3,5)(2,4,6),(1,6)(2,5) \rangle$ & $0$ \\   \hline
 105 & $S_4c=\langle (1,2,3)(5,6,7), (2,3)(4,5,6,7) \rangle$ & $0$ \\    \hline
105 & $S_4d=\langle  (1,3,2)(5,6,7), (2,3)(4,5,6,7) \rangle$ & $0$ \\ \hline
 70 & $A_4 \times C_3=\langle (1,3,5)(4,6,7), (1,2,3) \rangle$ & $0$ \\   \hline
 70 & $(C_3 \times C_3) \rtimes C_4=\langle (1,2,3),(4,5,6), (1,4)(2,5,3,6) \rangle$ & $\Z/2$ \\   \hline
 42 & $A_5a=\langle (1,2,3,4,5), (1,2,3) \rangle$ & $0$ \\   \hline
 42 & $A_5b= 	\langle (1,2,3,4,5), (1,4)(5,6) \rangle$ & $0$ \\   \hline
 35 & $(A_4 \times C_3) \rtimes C_2=\langle  (2,3)(5,7), (1,2)(4,5,6,7), (2,3)(5,6) \rangle$ & $0$ \\   \hline
 21 & $S_5=\langle  (1,2)(3,7), (2,6,5,4)(3,7) \rangle$ & $0$ \\   \hline
 15 & $\operatorname{PSL}(3,2)a=\langle (1,4)(2,3), (2,4,6)(3,5,7) \rangle$ & $0$ \\   \hline
 15 & $\operatorname{PSL}(3,2)b=\langle  (1,3)(2,7), (1,5,7)(3,4,6)  \rangle$ & $0$ \\   \hline
 7 & $A_6=\langle (1,2,3,4,5), (4,5,6)   \rangle$ & $0$ \\   \hline
\end{tabular}
}
\end{table}

\end{document}